\documentclass[letter,
fontsize=12pt,%
oneside,%
numbers=enddot]{scrartcl}
\KOMAoptions{DIV=14}    
\usepackage[utf8]{inputenc}
\usepackage[T1]{fontenc}
\usepackage{graphicx}
\usepackage{textcomp}
\usepackage[fleqn]{amsmath}
\usepackage{amsthm}
\usepackage{amssymb}
\usepackage{amsfonts}
\usepackage{accents}
\usepackage{soul}
\usepackage{pifont}        
\usepackage{libertine}
\usepackage[libertine,
slantedGreek,
nosymbolsc,
nonewtxmathopt,
subscriptcorrection]{newtxmath}
\usepackage[scaled=0.95,varqu,varl]{inconsolata}
\usepackage[scr=boondox]{mathalpha}   
\usepackage{euscript}   
\usepackage{leftindex}
\allowdisplaybreaks[1]  
\numberwithin{equation}{section}    
\usepackage[svgnames,hyperref]{xcolor}
\definecolor{orng}{HTML}{F35400}
\definecolor{bleu}{HTML}{BCE6F2}
\definecolor{dblue}{HTML}{0455BF}
\definecolor{dgreen}{HTML}{02724A}
\definecolor{dgreen2}{HTML}{025951}
\definecolor{dred}{HTML}{D90404}
\definecolor{dviolet}{HTML}{42208C}
\definecolor{labelkey}{HTML}{025951}
\definecolor{refkey}{HTML}{025951}
\definecolor{refkey}{rgb}{0,0.6,0.0}
\definecolor{Brown}{rgb}{0.45,0.0,0.05}
\definecolor{dgreen}{rgb}{0.00,0.49,0.00}
\definecolor{dblue}{rgb}{0,0.18,0.75}
\definecolor{lblue}{rgb}{0,0.7,0.75}
\definecolor{dviolet}{HTML}{9400D3}
\definecolor{pblue}{rgb}{0.1176,0.5647,1}
\definecolor{nblue}{rgb}{0.2,0.3,1}
\definecolor{pgreen}{rgb}{0.1961,0.8039,0.1961}
\definecolor{ngreen}{rgb}{0.0,0.6,0.3}
\definecolor{pred}{rgb}{1.0,0.2706,0.0}
\definecolor{magenta}{HTML}{ff00ff}
\definecolor{hotmagenta}{rgb}{1.0, 0.11, 0.81}
\definecolor{dorng}{rgb}{0.91,0.41,0.17}
\definecolor{dgray}{rgb}{0.41,0.41,0.41}
\definecolor{azure}{rgb}{0.0, 0.5, 1.0}
\usepackage{tikz,tkz-euclide,pgfplots}
\usetikzlibrary{arrows,decorations}
\addtokomafont{section}{\centering}

\usepackage{enumitem}
\setlist{itemsep=-2.0pt}

\usepackage{upref}  
\usepackage{hyperref}
\hypersetup{colorlinks=true,
linktocpage=true,
linkcolor=dblue,
citecolor=dgreen,
urlcolor=dred,
pdfencoding=auto,
hypertexnames=false}
\makeatletter
\g@addto@macro\th@plain{
\thm@headfont{\bfseries\sffamily}
\thm@notefont{}}
\g@addto@macro\th@definition{
\thm@headfont{\bfseries\sffamily}
\thm@notefont{}}
\g@addto@macro\th@remark{
\thm@headfont{\bfseries\sffamily}
\thm@notefont{}}
\makeatother
\theoremstyle{plain}
\newtheorem{theorem}{Theorem}[section]

\newtheorem{corollary}[theorem]{Corollary}
\newtheorem{lemma}[theorem]{Lemma}

\theoremstyle{definition}
\newtheorem{definition}[theorem]{Definition}

\theoremstyle{remark}
\newtheorem{remark}[theorem]{Remark}

\usepackage{epsfig}
\usepackage[usenames,dvipsnames]{pstricks}
\usepackage{pstricks-add}
\usepackage{pst-grad}
\usepackage{pst-plot} 
\usepackage{pst-node} 
\usepackage{pst-eucl} 
\usepackage{pst-coil} 
\usepackage[extdef=true]{delimset}
\DeclareMathDelimiterSet{\scal}[2]{
\selectdelim[l]<{#1}
\mathpunct{}\selectdelim[p]|
{#2}\selectdelim[r]>}
\DeclareMathDelimiterSet{\EC}[2]{
\mathsf{E}\selectdelim[l]({#1}
\mathpunct{}\selectdelim[p]|
{#2}\selectdelim[r])}

\newcommand{\menge}[2]{\bigl\{{#1}\mid{#2}\bigr\}} 
\DeclareMathDelimiterSet{\Menge}[2]{\selectdelim[l]\{
{#1}\selectdelim[m]|{#2}\selectdelim[r]\}}

\makeatletter
\def\upintkern@{\mkern-7mu\mathchoice{\mkern-3.5mu}{}{}{}}
\def\upintdots@{\mathchoice{\mkern-4mu\@cdots\mkern-4mu}%
{{\cdotp}\mkern1.5mu{\cdotp}\mkern1.5mu{\cdotp}}%
{{\cdotp}\mkern1mu{\cdotp}\mkern1mu{\cdotp}}%
{{\cdotp}\mkern1mu{\cdotp}\mkern1mu{\cdotp}}}
\makeatother
\DeclareFontFamily{OMX}{mdbch}{}
\DeclareFontShape{OMX}{mdbch}{m}{n}{ <->s * [0.8]  mdbchr7v }{}
\DeclareFontShape{OMX}{mdbch}{b}{n}{ <->s * [0.8]  mdbchb7v }{}
\DeclareFontShape{OMX}{mdbch}{bx}{n}{<->ssub * mdbch/b/n}{}
\DeclareSymbolFont{uplargesymbols}{OMX}{mdbch}{m}{n}
\SetSymbolFont{uplargesymbols}{bold}{OMX}{mdbch}{b}{n}
\DeclareMathSymbol{\upintop}{\mathop}{uplargesymbols}{82}
\DeclareMathSymbol{\upointop}{\mathop}{uplargesymbols}{"48}
\makeatletter
\renewcommand{\int}{\DOTSI\upintop\ilimits@}
\renewcommand{\oint}{\DOTSI\upointop\ilimits@}
\makeatother

\newcommand{\bunder}[2][4]{\mathrlap{\mkern\the\numexpr#1/2mu%
\relax\underline{\phantom{\mathrm{#2}\mkern-#1mu}}}#2}
\newcommand{\RR}{\mathbb{R}}
\newcommand{\NN}{\mathbb{N}}

\newcommand{\AS}{\mathsf{A}}
\newcommand{\fS}{\mathsf{f}}

\newcommand{\TS}{\mathsf{T}}

\newcommand{\HS}{\mathsf{H}}

\newcommand{\zS}{\mathsf{z}}
\newcommand{\nnn}{\mathsf{n}\in\mathbb{N}}
\newcommand{\jjj}{\mathsf{j}\in\mathbb{N}}

\newcommand{\nS}{{\mathsf{n}}}

\newcommand{\jS}{{\mathsf{j}}}
\newcommand{\kS}{{\mathsf{k}}}
\newcommand{\xS}{{\mathsf{x}}}
\newcommand{\yS}{{\mathsf{y}}}

\newcommand{\BE}{\EuScript{B}}
\newcommand{\FE}{\EuScript{F}}

\newcommand{\elll}{\ensuremath{\mbox{\large$\ell$}}}

\newcommand{\pinf}{{+}\infty}
\newcommand{\minf}{{-}\infty}
\newcommand{\zeroun}{\intv[o]{0}{1}}

\newcommand{\RX}{\intv[l]0{\minf}{\pinf}}
\newcommand{\RP}{\intv[r]0{0}{\pinf}}
\newcommand{\RPP}{\intv[o]0{0}{\pinf}}

\newcommand{\emp}{\varnothing}

\newcommand{\Int}{\displaystyle\int}

\newcommand{\pushfwd}%
{\ensuremath{\mbox{\Large$\,\triangleright\,$}}}

\DeclareMathOperator{\Argmin}{Argmin}
\DeclareMathOperator{\argmin}{argmin}
\newcommand{\Id}{\mathsf{Id}}

\newcommand{\moyo}[2]{\leftindex[I]^{#2}{#1}}

\DeclareMathOperator{\dom}{dom}

\DeclareMathOperator{\Fix}{Fix}
\DeclareMathOperator{\gra}{gra}
\DeclareMathOperator{\zer}{zer}

\DeclareMathOperator{\prox}{prox}

\newcommand{\EE}{\mathsf{E}}
\newcommand{\PP}{\mathsf{P}}

\renewcommand{\leq}{\leqslant}
\renewcommand{\geq}{\geqslant}

\newcommand{\weakly}{\rightharpoonup}

\newcommand{\Pas}{\text{\normalfont$\PP$-a.s.}}

\renewenvironment{abstract}{%
\vspace*{-0.50cm}
\small
\quotation%
\noindent%
{\normalfont\bfseries\sffamily
\nobreak\abstractname\ }%
}{%
\endquotation%
\medskip
}
\renewcommand{\abstractname}{Abstract.}
\newcommand\keywordsname{Keywords.}
\newenvironment{keywords}
{\renewcommand\abstractname{\keywordsname}\begin{abstract}}
{\end{abstract}}

\usepackage[auth-sc]{authblk}
\newcommand{\email}[1]{\href{mailto:#1}{\nolinkurl{#1}}}
\renewcommand*\Affilfont{\normalfont\normalsize}
\newcommand\affilcr{\protect\\ \protect\Affilfont}
\makeatletter
\renewcommand\AB@affilsepx{\protect\\[0.5em]}
\makeatother

\author[1]{Patrick L. Combettes}
\affil[1]{North Carolina State University
\affilcr
Department of Mathematics
\affilcr
Raleigh, NC 27695, USA
\affilcr
\email{plc@math.ncsu.edu}
}
\author[2]{Javier I. Madariaga}
\affil[2]{North Carolina State University
\affilcr
Department of Mathematics
\affilcr
Raleigh, NC 27695, USA
\affilcr
\email{jimadari@ncsu.edu}
}

\begin{document}

\title{ 
Almost Supermartingale Extensions of\\
Olivier's Theorem\thanks{Contact author: 
P. L. Combettes. Email: \email{plc@math.ncsu.edu}.
Phone: +1 919 515 2671.
This work was supported by the National
Science Foundation under grant DMS-2513409.
}}

\date{~}

\maketitle

\thispagestyle{empty}
\begin{abstract} 
Olivier's 1827 theorem provides a rate of convergence to zero of
the general term of a decreasing summable sequence of positive
reals. We derive stochastic extensions of this result in the
context of almost supermartingales. The results are applied to
the analysis of stochastic iterative processes.
\end{abstract}

\begin{keywords}
Olivier's theorem, 
Robbins--Siegmund theorem, 
stochastic gradient method, 
stochastic Krasnosel'ski\u\i--Mann iterations,
stochastic proximal point algorithm.
\end{keywords}

\newpage

\section{Introduction}
\label{sec:1}

In 1827, Olivier \cite{Oliv27} argued that a decreasing
sequence $(\upxi_\nS)_{\nS\in\NN}$ in $\RP$ is summable if and only
if $\nS\upxi_\nS\to 0$. The sufficiency statement was shown to be
wrong by Abel \cite{Abel28}, who proposed as a counterexample
$(\forall\nnn)$ $\upxi_{\nS}=1/((\nS+2)\ln(\nS+2))$. Nonetheless, 
the necessity statement is correct and it constitutes what is now
known as Olivier's theorem.

\begin{theorem}[Olivier]
\label{t:o}
Let $(\upxi_\nS)_{\nnn}$ be a summable decreasing sequence in 
$\RP$. Then $\upxi_\nS=\mathrm{o}(1/\nS)$.
\end{theorem}

An improvement of Olivier's theorem that includes weights in the
series can be found in \cite{Vall14}.

\begin{lemma}[\protect{\cite[p.~408]{Vall14}}]
\label{l:2}
Let $(\upalpha_\nS)_{\nnn}$ be a sequence in $\RPP$ such that
$\sum_{\nnn}\upalpha_{\nS}=\pinf$ and define $(\forall\nnn)$
$\uptheta_\nS=\sum_{\kS=0}^{\nS}\upalpha_\kS$. Further, let
$(\upxi_\nS)_{\nnn}$ be a sequence in $\RP$ such that 
$(\upxi_\nS)_{\nnn}$ is decreasing and 
$\sum_{\nnn}\upalpha_\nS\upxi_{\nS}<\pinf$. Then
$\upxi_\nS=\mathrm{o}(1/\uptheta_\nS)$.
\end{lemma}

Another extension of Olivier's theorem replaces the monotonicity
property of $(\upxi_\nS)_{\nnn}$ by a quasi-monotonicity property
which goes back to \cite{Szas48}.

\begin{lemma}[\protect{\cite[Theorem~1]{Robe68}}]
\label{l:7}
Let $(\upalpha_\nS)_{\nnn}$ be a sequence in $\RPP$ such that
$\sum_{\nnn}\upalpha_{\nS}=\pinf$ and define $(\forall\nnn)$
$\uptheta_\nS=\sum_{\kS=0}^{\nS}\upalpha_\kS$. Further, let
$(\upxi_\nS)_{\nnn}$ and $(\upvarepsilon_\nS)_{\nnn}$ be sequences
in $\RP$ such that 
$\sum_{\nnn}\upalpha_\nS\upxi_{\nS}<\pinf$,
$\sum_{\nnn}\uptheta_{\nS}\upvarepsilon_{\nS}<\pinf$, and
$(\forall\nnn)$ $\upxi_{\nS+1}\leq
\upxi_{\nS}+\upvarepsilon_{\nS}$.
Then $\upxi_\nS=\mathrm{o}(1/\uptheta_{\nS})$.
\end{lemma}

Olivier's theorem and its above variants have found a variety of
applications in applied analysis, e.g.,
\cite{Boas65,Svva21,MaPa18,Davi16,Lian16,Robe68,Sing78}. The
purpose of this paper is to derive stochastic versions of Olivier's
theorem to address applications based on probabilistic models. In
this respect, our main focus is on obtaining new convergence rates
for various stochastic iteration methods. In particular, our
analysis is applied to a stochastic version of the Br\'ezis--Lions
proximal point algorithm for finding zeros of monotone operators, a
proximal point algorithm with a new stochastic approximation scheme
for minimizing convex functions, a randomly relaxed
Krasnosel'ski\u\i--Mann method with stochastic approximations for
finding a fixed point of nonexpansive operators, and the stochastic
gradient method for convex minimization under broadened
assumptions.

Throughout, the underlying probability space 
$(\upOmega,\FE,\PP)$ is assumed to be complete. Let
$\mathfrak{F}=(\FE_{\nS})_{\nnn}$ be a filtration of $\FE$ and let
$\mathsf{S}$ be a Borel subset of $\RR$. Then
$\elll(\mathfrak{F};\mathsf{S})$ is the
set of sequences of $\mathsf{S}$-valued random variables
$(\alpha_{\nS})_{\nnn}$ which are adapted to $\mathfrak{F}$, i.e.,
for every $\nnn$, $\alpha_{\nS}$ is $\FE_{\nS}$-measurable. We set
\begin{equation}
\elll^1(\mathfrak{F};\mathsf{S})=
\Menge3{(\alpha_{\nS})_{\nnn}\in\elll(\mathfrak{F};\mathsf{S})}
{\sum_{\nnn}\abs{\alpha_{\nS}}<\pinf\;\Pas}.
\end{equation}

Our main result is the following.

\begin{theorem}
\label{t:1}
Let $\mathfrak{F}=(\FE_{\nS})_{\nnn}$ be a filtration of $\FE$. Let
$(\alpha_{\nS})_{\nnn}\in\elll(\mathfrak{F};\RPP)$ be such that
$\sum_{\nnn}\alpha_{\nS}=\pinf\;\Pas$ and define $(\forall\nnn)$
$\theta_\nS=\sum_{\kS=0}^{\nS}\alpha_\kS$. Further, 
let $(\xi_{\nS})_{\nnn}\in\elll(\mathfrak{F};\RP)$, 
$(\delta_{\nS})_{\nnn}\in\elll^1(\mathfrak{F};\RP)$, and
$(\varepsilon_{\nS})_{\nnn}\in\elll(\mathfrak{F};\RP)$.
Suppose that the almost supermartingale property
\begin{equation}
\label{e:c1}
(\forall\nnn)\quad\EC{\xi_{\nS+1}}{\FE_{\nS}}
\leq(1+\delta_{\nS})\xi_{\nS}+\varepsilon_{\nS}\;\Pas
\end{equation}
is satisfied and that
$\sum_{\nnn}\theta_{\nS}\varepsilon_{\nS}<\pinf\;\Pas$
Then the following hold:
\begin{enumerate}
\item
\label{t:1i}
Suppose that 
$\sum_{\nnn}\alpha_{\nS}\xi_{\nS}<\pinf\;\Pas$
Then $\xi_{\nS+1}=\mathrm{o}(1/\theta_\nS)\;\Pas$
\item
\label{t:1ii}
Suppose that 
$\sum_{\nnn}\EC{\alpha_{\nS+1}\xi_{\nS+1}}{\FE_\nS}<\pinf\;\Pas$
Then $\xi_{\nS}=\mathrm{o}(1/\theta_\nS)\;\Pas$
\item
\label{t:1iii}
Suppose that 
$\sum_{\nnn}\EE\brk{\theta_{\nS}\varepsilon_{\nS}}<\pinf$,
$\sum_{\nnn}\EE\brk{\alpha_{\nS}\xi_{\nS}}<\pinf$,
$\sum_{\nnn}\EE\delta_{\nS}<\pinf$,
and, for every $\nnn$, $\delta_{\nS}$ is independent of
$(\xi_{\nS},\theta_{\nS})$.
Then $\xi_{\nS}=\mathrm{o}(1/\theta_\nS)\;\Pas$ and
$\EE{\sqrt{\xi_{\nS}\theta_{\nS}}}\to 0$.
\end{enumerate}
\end{theorem}

Section~\ref{sec:2} is dedicated to proving Theorem~\ref{t:1} and
it also features special cases of interest. Applications to the
analysis of stochastic iteration schemes are presented in
Section~\ref{sec:3}. Throughout, we use sans-serif letters to
denote deterministic variables and italicized serif letters to
denote random variables. 

\section{Proof and special cases}
\label{sec:2}

\subsection{Preliminary results}
\label{sec:21}

In his study of Olivier's result \cite{Oliv27}, Abel \cite{Abel28}
proved the following property.

\begin{lemma}[Abel]
\label{l:1}
Let $(\upalpha_\nS)_{\nnn}$ be a sequence in $\RPP$ such that
$\sum_{\nnn}\upalpha_\nS=\pinf$ and define $(\forall\nnn)$
$\uptheta_\nS=\sum_{\kS=0}^{\nS}\upalpha_\kS$. Then
$\sum_{\nnn}\upalpha_{\nS+1}/\uptheta_\nS=\pinf$.
\end{lemma}

The following sharpening of Lemma~\ref{l:1} by Dini
\cite{Dini67} is known as the Abel--Dini theorem.

\begin{lemma}[Abel--Dini]
\label{l:4}
Let $(\upalpha_\nS)_{\nnn}$ be a sequence in $\RPP$ such that
$\sum_{\nnn}\upalpha_\nS=\pinf$ and define $(\forall\nnn)$
$\uptheta_\nS=\sum_{\kS=0}^{\nS}\upalpha_\kS$. Then
$\sum_{\nnn}\upalpha_{\nS}/\uptheta_\nS=\pinf$.
\end{lemma}

The Robbins--Siegmund theorem \cite[Theorem~1]{Robb71} is an
essential tool in the study of the asymptotic behavior of
stochastic iterations \cite{Siop15,Dufl90,Fran22,Neri26}; see also
\cite{Ermo69,Glad65} for anterior special cases. 

\begin{lemma}[Robbins--Siegmund]
\label{l:5}
Let $\mathfrak{F}=(\FE_{\nS})_{\nnn}$ be a filtration of $\FE$.
Let $(\beta_{\nS})_{\nnn}\in\elll(\mathfrak{F};\RP)$,
$(\eta_{\nS})_{\nnn}\in\elll(\mathfrak{F};\RP)$, 
$(\delta_{\nS})_{\nnn}\in\elll^1(\mathfrak{F};\RP)$, and
$(\tau_{\nS})_{\nnn}\in\elll^1(\mathfrak{F};\RP)$ satisfying the
almost supermartingale property
\begin{equation}
(\forall\nnn)\quad\EC{\beta_{\nS+1}}{\FE_{\nS}}
+\eta_{\nS}\leq(1+\delta_{\nS})\beta_{\nS}+\tau_{\nS}\;\Pas
\end{equation}
Then $(\eta_{\nS})_{\nnn}\in\elll^1(\mathfrak{F};\RP)$ and
$(\beta_{\nS})_{\nnn}$ converges $\Pas$ to a $\RP$-valued random
variable.
\end{lemma}

\begin{corollary}
\label{c:3}
Let $(\upbeta_{\nS})_{\nnn}$, $(\upeta_{\nS})_{\nnn}$, and
$(\uptau_{\nS})_{\nnn}$ be sequences in $\RP$ such that
$\sum_{\nnn}\uptau_{\nS}<\pinf$ and $(\forall\nnn)$ 
$\upbeta_{\nS+1}+\upeta_{\nS}\leq\upbeta_{\nS}+\uptau_{\nS}$. Then 
$\sum_{\nnn}\upeta_{\nS}<\pinf$ and $(\upbeta_{\nS})_{\nnn}$
converges.
\end{corollary}

\begin{corollary}
\label{c:4}
Let $\mathfrak{F}=(\FE_{\nS})_{\nnn}$ be a filtration of $\FE$ and
let $(\xi_{\nS})_{\nnn}\in\elll(\mathfrak{F};\RP)$ be such that
$\sum_{\nnn}\EC{\xi_{\nS+1}}{\FE_{\nS}}<\pinf\;\Pas$ Then
$\sum_{\nnn}\xi_{\nS}<\pinf\;\Pas$
\end{corollary}
\begin{proof}
Set $(\forall\nnn)$ $\beta_{\nS}=\sum_{\kS=0}^{\nS}\xi_k$. Then
$(\forall\nnn)$ $\EC{\beta_{\nS+1}}{\FE_{\nS}}=
\beta_{\nS}+\EC{\xi_{\nS+1}}{\FE_{\nS}}$. Lemma~\ref{l:5} then
ensures that $(\beta_{\nS})_{\nnn}$ converges $\Pas$, that is,
$\sum_{\nnn}\xi_{\nS}<\pinf\;\Pas$
\end{proof}

The following fact will also be needed.

\begin{lemma}[\protect{\cite[Lemma~2.7]{Moco26}}]
\label{l:6}
Let $\xi\in L^1(\upOmega,\FE,\PP;\RR)$, let $\upPhi$ be a family of
random variables, let $\EuScript{A}$ be the sub-$\upsigma$-algebra 
of $\FE$ generated by $\upPhi$, and let 
$\eta\in L^1(\upOmega,\FE,\PP;\RR)$ be independent of 
the sub-$\upsigma$-algebra generated by $\{\xi\}\cup\upPhi$. Then
$\EC{\eta\xi}{\EuScript{A}}=\EE\eta\EC{\xi}{\EuScript{A}}$.
\end{lemma}

\subsection{Proof of \protect{Theorem~\ref{t:1}}}

\begin{proof}
We observe that, for every $\nnn$, $\theta_{\nS}$ is
$\FE_{\nS}$-measurable.

\ref{t:1i}: Set 
\begin{equation}
\beta_0=0\quad\text{and}\quad(\forall\nnn)\quad
\begin{cases}
\beta_{\nS+1}=\xi_{\nS+1}\theta_{\nS}\\
\tau_\nS=\brk{1+\delta_{\nS}}\alpha_{\nS}\xi_{\nS}
+\varepsilon_{\nS}\theta_{\nS}.
\end{cases}
\end{equation}
Then we must show that
\begin{equation}
\label{e:100}
\beta_\nS\to 0\;\,\Pas
\end{equation}
We derive from \eqref{e:c1} that
\begin{align}
(\forall\nnn)\quad\EC1{\beta_{\nS+1}}{\FE_{\nS}}
&=\EC1{\xi_{\nS+1}}{\FE_{\nS}}\theta_{\nS}\nonumber\\
&\leq\brk1{1+\delta_{\nS}}\xi_{\nS}\theta_{\nS}
+\varepsilon_{\nS}\theta_{\nS}
\nonumber\\
&=\brk1{1+\delta_{\nS}}\brk1{\beta_{\nS}+\alpha_{\nS}\xi_{\nS}}
+\varepsilon_{\nS}\theta_{\nS}
\nonumber\\
&=\brk1{1+\delta_{\nS}}\beta_{\nS}
+\brk{1+\delta_{\nS}}\alpha_{\nS}\xi_{\nS}
+\varepsilon_{\nS}\theta_{\nS}\nonumber\\
&=\brk1{1+\delta_{\nS}}\beta_{\nS}+\tau_{\nS}\;\;\Pas,
\end{align}
where it follows from the assumptions that
$(\tau_{\nS})_{\nnn}\in\elll^1(\mathfrak{F};\RP)$. Therefore 
Lemma~\ref{l:5} asserts that $(\beta_{\nS})_{\nnn}$ converges
$\Pas$ to a $\RP$-valued random variable. In view of \eqref{e:100},
to complete the proof, it is enough to show that
\begin{equation}
\label{e:09}
\varliminf\beta_{\nS}=0\;\,\Pas
\end{equation}
Suppose, to the contrary, that there exists $\upOmega'\in\FE$ such
that $\PP(\upOmega')>0$ and, for every $\upomega\in\upOmega'$,
$\sum_{\nnn}\alpha_{\nS}(\upomega)\xi_{\nS}(\upomega)<\pinf$,
$\sum_{\nnn}\alpha_{\nS}(\upomega)=\pinf$, and there exist
$\upepsilon\in\RPP$ and $\mathrm{N}\in\NN$ such that, for every
integer $\nS\geq\mathrm{N}$,
$\beta_{\nS+1}(\upomega)\geq\upepsilon$, that is,
$\xi_{\nS+1}(\upomega)\geq\upepsilon/\theta_{\nS}(\upomega)$ and,
therefore, 
\begin{equation}\alpha_{\nS+1}(\upomega)\xi_{\nS+1}(\upomega)\geq
\upepsilon\alpha_{\nS+1}(\upomega)/\theta_{\nS}(\upomega).
\end{equation}
Altogether, upon invoking Lemma~\ref{l:1}, we arrive at the
contradiction
\begin{equation}
(\forall\upomega\in\upOmega')\quad
\pinf>\sum_{\nS\geq\mathrm{N}}\alpha_{\nS+1}(\upomega)
\xi_{\nS+1}(\upomega)\geq\upepsilon\sum_{\nS\geq \mathrm{N}}
\dfrac{\alpha_{\nS+1}(\upomega)}{\theta_{\nS}(\upomega)}=\pinf,
\end{equation}
and hence conclude that \eqref{e:09} holds.

\ref{t:1ii}: Set 
\begin{equation}
\label{e:set1}
(\forall\nnn)\quad
\begin{cases}
\beta_{\nS}=\xi_{\nS}\theta_{\nS}\\
\tau_\nS=\varepsilon_{\nS}\theta_{\nS}
+\EC1{\alpha_{\nS+1}\xi_{\nS+1}}{\FE_{\nS}}.
\end{cases}
\end{equation}
To show that $\beta_\nS\to 0$ $\Pas$,
we first note that \eqref{e:c1} yields
\begin{align}
(\forall\nnn)\quad\EC1{\beta_{\nS+1}}{\FE_{\nS}}
&=\EC1{\theta_{\nS+1}\xi_{\nS+1}}{\FE_{\nS}}\nonumber\\
&=\EC{\xi_{\nS+1}}{\FE_{\nS}}\theta_{\nS}
+\EC1{\alpha_{\nS+1}\xi_{\nS+1}}{\FE_{\nS}}\nonumber\\
&\leq\brk1{1+\delta_{\nS}}\xi_{\nS}\theta_{\nS}
+\varepsilon_{\nS}\theta_{\nS}
+\EC1{\alpha_{\nS+1}\xi_{\nS+1}}{\FE_{\nS}}
\nonumber\\
&=\brk1{1+\delta_{\nS}}\beta_{\nS}
+\varepsilon_{\nS}\theta_{\nS}
+\EC1{\alpha_{\nS+1}\xi_{\nS+1}}{\FE_{\nS}}\nonumber\\
&=\brk1{1+\delta_{\nS}}\beta_{\nS}+\tau_{\nS}\;\;\Pas
\label{e:878}
\end{align}
Since $(\tau_{\nS})_{\nnn}\in\elll^1(\mathfrak{F};\RP)$, 
we deduce from Lemma~\ref{l:5} that 
$(\beta_{\nS})_{\nnn}$ converges $\Pas$ to a $\RP$-valued random
variable. Hence, it remains to show that
\begin{equation}
\label{e:092}
\varliminf\beta_{\nS}=0\;\,\Pas
\end{equation}
We argue by contradiction. First, we invoke Corollary~\ref{c:4}
to get $\sum_{\nnn}\alpha_{\nS}\xi_{\nS}<\pinf\;\Pas$
Next, suppose that there exists $\upOmega'\in\FE$ such
that $\PP(\upOmega')>0$ and, for every $\upomega\in\upOmega'$,
$\sum_{\nnn}\alpha_{\nS}(\upomega)\xi_{\nS}(\upomega)<\pinf$, 
$\sum_{\nnn}\alpha_{\nS}(\upomega)=\pinf$, and
there exist $\upepsilon\in\RPP$ and $\mathrm{N}\in\NN$ such that,
for every integer $\nS\geq\mathrm{N}$,
$\beta_{\nS}(\upomega)\geq\upepsilon$, hence,
$\alpha_{\nS}(\upomega)\xi_{\nS}(\upomega)\geq
\upepsilon\alpha_{\nS}(\upomega)/\theta_{\nS}(\upomega)$.
Appealing to Lemma~\ref{l:4}, we obtain the contradiction
\begin{equation}
(\forall\upomega\in\upOmega')\quad
\pinf>\sum_{\nS\geq\mathrm{N}}\alpha_{\nS}(\upomega)
\xi_{\nS}(\upomega)\geq\upepsilon\sum_{\nS\geq \mathrm{N}}
\dfrac{\alpha_{\nS}(\upomega)}{\theta_{\nS}(\upomega)}=\pinf.
\end{equation}
This confirms that \eqref{e:092} holds.

\ref{t:1iii}: The assumptions imply that
$\sum_{\nnn}\varepsilon_{\nS}\theta_{\nS}<\pinf\;\Pas$
and that
$\sum_{\nnn}\EC{\alpha_{\nS+1}\xi_{\nS+1}}{\FE_\nS}<\pinf\;\Pas$
Hence it follows from \ref{t:1ii} that 
$\xi_{\nS}=\mathrm{o}(1/\theta_\nS)\;\Pas$ 
Now fix $\nnn$. Then the $\FE_{\nS}$-measurability
of $\theta_{\nS}$, the independence of $\delta_{\nS}$
from $(\xi_{\nS},\theta_{\nS})$, and \eqref{e:c1} yield
\begin{equation}
\label{e:301}
\EE\brk1{\theta_{\nS+1}\xi_{\nS+1}}
=\EE\brk1{\theta_{\nS}\xi_{\nS+1}}
+\EE\brk1{\alpha_{\nS+1}\xi_{\nS+1}}
\leq\brk1{1+\EE\delta_{\nS}}\EE\brk1{\theta_{\nS}\xi_{\nS}}
+\EE\brk1{\alpha_{\nS+1}\xi_{\nS+1}}
+\EE\brk1{\theta_{\nS}\varepsilon_{\nS}}.
\end{equation}
The assumptions imply that 
$\EE(\theta_{0}\xi_{0})=\EE(\alpha_{0}\xi_{0})<\pinf$, which
recursively ensures through \eqref{e:301} that, for every $\nnn$,
$\EE(\theta_{\nS}\xi_{\nS})<\pinf$. 
Moreover, the assumptions and Lemma~\ref{l:5} (applied to the
natural filtration) guarantee that
$(\EE(\theta_{\nS}\xi_{\nS}))_{\nnn}$
converges to a positive real number and, hence, that it is bounded.
Now set, for every $\nnn$,
$\chi_{\nS}=\sqrt{\xi_{\nS}\theta_{\nS}}$. Then
\begin{equation}
\chi_{\nS}\to 0\;\;\Pas\quad\text{and}\quad
\sup_{\nnn}\EE\abs{\chi_{\nS}}^2<\pinf.
\end{equation}
We deduce from \cite[Lemma~2.9]{Moco26} that 
$\EE\chi_{\nS}=\EE{\sqrt{\xi_{\nS}\theta_{\nS}}}\to 0$, which
completes the proof.
\end{proof}

\begin{remark}
\label{r:1}
A key point in the above proof is that 
\begin{equation}
\label{e:69}
(\beta_{\nS})_{\nnn}\;\:\text{converges almost surely}. 
\end{equation}
We use the almost supermartingale 
property \eqref{e:c1} to secure \eqref{e:69} as it manifests
itself in many models of interest. However, it is clear that 
the conclusions of Theorem~\ref{t:1} remain valid under the more
general assumption \eqref{e:69}.
\end{remark}

\subsection{Special cases}

A noteworthy special case of Theorem~\ref{t:1} is the following. In
the deterministic setting and in the absence of the sequence
$(\delta_{\nS})_{\nnn}$, it reduces to \cite[Theorem~1]{Boas65}. 

\begin{corollary}
\label{c:1}
Let $\mathfrak{F}=(\FE_{\nS})_{\nnn}$ be a filtration of $\FE$ and
let $\mathsf{p}\in\left]1,\pinf\right[$. 
Further, let $(\xi_{\nS})_{\nnn}\in\elll(\mathfrak{F};\RP)$, 
$(\delta_{\nS})_{\nnn}\in\elll^1(\mathfrak{F};\RP)$, and
$(\varepsilon_{\nS})_{\nnn}\in\elll(\mathfrak{F};\RP)$.
Suppose that the almost supermartingale property
\begin{equation}
(\forall\nnn)\quad\EC{\xi_{\nS+1}}{\FE_{\nS}}
\leq(1+\delta_{\nS})\xi_{\nS}+\varepsilon_{\nS}\;\Pas
\end{equation}
is satisfied and that
$\sum_{\nnn}\nS^{\mathsf{p}}\varepsilon_{\nS}<\pinf\;\Pas$
Then the following hold:
\begin{enumerate}
\item
\label{c:1i}
Suppose that 
$\sum_{\nnn}\nS^{\mathsf{p}-1}\xi_{\nS}<\pinf\;\Pas$
Then $\xi_{\nS+1}=\mathrm{o}(1/\nS^{\mathsf{p}})\;\Pas$
\item
\label{c:1ii}
Suppose that 
$\sum_{\nnn}\brk{\nS+1}^{\mathsf{p}-1}
\EC{\xi_{\nS+1}}{\FE_\nS}<\pinf\;\Pas$
Then $\xi_{\nS}=\mathrm{o}(1/\nS^{\mathsf{p}})\;\Pas$
\item
\label{c:1iii}
Suppose that 
$\sum_{\nnn}\nS^{\mathsf{p}}\EE\varepsilon_{\nS}<\pinf$,
$\sum_{\nnn}\nS^{\mathsf{p}-1}\EE\xi_{\nS}<\pinf$,
$\sum_{\nnn}\EE\delta_{\nS}<\pinf$,
and, for every $\nnn$, $\delta_{\nS}$ is independent of
$\xi_{\nS}$.
Then $\xi_{\nS}=\mathrm{o}(1/\nS^{\mathsf{p}})\;\Pas$ and
$\EE{\sqrt{\xi_{\nS}}}=\mathrm{o}(1/\nS^{\mathsf{p}/2})$.
\end{enumerate}
\end{corollary}
\begin{proof}
Apply Theorem~\ref{t:1} with $(\forall\nnn)$
$\alpha_{\nS}=\nS^{\mathsf{p}-1}$.
\end{proof}

Concerning the deterministic setting, let us record an extension of
Lemma~\ref{l:7}.

\begin{corollary}
\label{c:2}
Let $(\upalpha_{\nS})_{\nnn}$ be a sequence in $\RPP$ such that
$\sum_{\nnn}\upalpha_{\nS}=\pinf$ and define $(\forall\nnn)$ 
$\uptheta_{\nS}=\sum_{\kS=0}^{\nS}\upalpha_{\kS}$. Further, let
$(\upxi_{\nS})_{\nnn}$, $(\updelta_{\nS})_{\nnn}$, and
$(\upvarepsilon_{\nS})_{\nnn}$ be sequences in $\RP$ such that
\begin{equation}
\label{e:c2}
\sum_{\nnn}\updelta_{\nS}<\pinf,\; 
\sum_{\nnn}\upalpha_{\nS}\upxi_{\nS}<\pinf,\;
\sum_{\nnn}\uptheta_{\nS}\upvarepsilon_{\nS}<\pinf,
\;\;\text{and}\;\;
(\forall\nnn)\;\;\upxi_{\nS+1}
\leq(1+\updelta_{\nS})\upxi_{\nS}+\upvarepsilon_{\nS}.
\end{equation}
Then $\upxi_{\nS}=\mathsf{o}\brk{1/\uptheta_{\nS}}.$  
\end{corollary}
\begin{proof}
Apply Theorem~\ref{t:1}\ref{t:1ii} with $(\forall\nnn)$
$\FE_{\nS}=\{\emp,\upOmega\}$.
\end{proof}

\section{Application to stochastic iterative methods}
\label{sec:3}

We apply Theorem~\ref{t:1} to the study of the asymptotic behavior
of various stochastic iterative methods to obtain new convergence
rates. 

\subsection{Notation and background}

Throughout, $\HS$ denotes a separable real Hilbert space with power
set $2^{\HS}$, scalar product $\scal{\cdot}{\cdot}$, associated
norm $\norm{\cdot}$, and associated Borel $\upsigma$-algebra
$\BE_{\HS}$. Given a sub-$\upsigma$-algebra $\EuScript{A}$ of
$\FE$, $L^2(\upOmega,\EuScript{A},\PP;\HS)$ is the space of
equivalence classes of $\Pas$ equal $\HS$-valued random variables
$x\colon(\upOmega,\EuScript{A},\PP)\to(\HS,\BE_{\HS})$ such that
$\EE\|x\|^2<\pinf$. Given a sequence $(x_{\nS})_{\nnn}$ in 
$L^2(\upOmega,\FE,\PP;\HS)$, the filtration $(\FE_{\nS})_{\nnn}$ of
$\FE$ under consideration will always be that obtained by letting,
for every $\nnn$, $\FE_{\nS}$ be the sub-$\upsigma$-algebra
generated by $(x_{0},\ldots,x_{\nS})$.

Let $\TS\colon\HS\to\HS$. The fixed point set of $\TS$ is
$\Fix\TS=\menge{\mathsf{x}\in\HS}{\TS\mathsf{x}=\mathsf{x}}$. 
Further, $\TS$ is nonexpansive if it is $1$-Lipschitzian and firmly
nonexpansive if
\begin{equation}
\label{e:95}
(\forall\xS\in\HS)(\forall\mathsf{y}\in\HS)\quad
\norm{\mathsf{T}\xS-\mathsf{T}\mathsf{y}}^2\leq
\norm{\xS-\mathsf{y}}^2-
\norm{\brk{\Id-\mathsf{T}}\xS-\brk{\Id-\mathsf{T}}\mathsf{y}}^2.
\end{equation}

Let $\AS\colon\HS\to 2^{\HS}$. The graph of $\AS$ is
$\gra\AS=\menge{(\xS,\xS^*)\in\HS\times\HS}{\xS^*\in\AS\xS}$ and
the set of zeros of $\AS$ is
$\zer\AS=\menge{\xS\in\HS}{0\in\AS\xS}$. The inverse of
$\AS$ is the operator $\AS^{-1}\colon\HS\to2^{\HS}$ with graph
$\gra\AS^{-1}=\menge{(\xS^*,\xS)\in\HS\times\HS}{\xS^*\in\AS\xS}$.
The resolvent of $\AS$ is $\mathsf{J}_{\AS}=(\Id+\AS)^{-1}$ and
the Yosida approximation of $\AS$ of index $\upgamma\in\RPP$ is
\begin{equation}
\moyo{\AS}{\upgamma}=
\dfrac{\Id-\mathsf{J}_{\upgamma\AS}}{\upgamma}.
\end{equation}
Suppose that $\AS$ is monotone, that is, 
\begin{equation}
\brk1{\forall(\xS,\xS^*)\in\gra\AS}
\brk1{\forall(\yS,\yS^*)\in\gra\AS}\quad
\scal{\xS-\yS}{\xS^*-\yS^*}\geq 0.
\end{equation}
Then $\AS$ is maximally monotone if there exists no
monotone operator $\mathsf{B}\colon\HS\to 2^{\HS}$ such that 
$\gra\AS\subset\gra\mathsf{B}\neq\gra\AS$. In this case,
$\dom\mathsf{J}_{\AS}=\HS$ and $\mathsf{J}_{\AS}$ is firmly
nonexpansive.

We denote by $\upGamma_0(\HS)$ the class of lower semicontinuous
convex functions $\mathsf{f}\colon\HS\to\RX$ which are not
identically $\pinf$. Let $\mathsf{f}\in\upGamma_0(\HS)$. The set of
minimizers of $\fS$ is $\Argmin\fS$, the subdifferential of
$\mathsf{f}$ is the maximally monotone operator
\begin{equation}
\uppartial\mathsf{f}\colon\HS\to 2^{\HS}\colon
\xS\mapsto\menge{\xS^*\in\HS}{(\forall\yS\in\HS)\;
\scal{\yS-\xS}{\xS^*}+\mathsf{f}(\xS)\leq\mathsf{f}(\yS)}, 
\end{equation}
and the proximity operator of $\mathsf{f}$ is 
\begin{equation}
\prox_{\mathsf{f}}=\mathsf{J}_{\uppartial\mathsf{f}}\colon\HS\to\HS
\colon\xS\mapsto\underset{\yS\in\HS}{\argmin}\brk2{\mathsf{f}(\yS)
+\dfrac{1}{2}\norm{\xS-\yS}^2}. 
\end{equation}
Alternatively, 
\begin{equation}
\label{e:94}
\brk1{\forall(\xS,\mathsf{p})\in\HS\times\HS}\quad
\mathsf{p}=\prox_{\fS}\xS\;\Leftrightarrow\;
\brk{\forall\yS\in\HS}\quad
\scal{\yS-\mathsf{p}}{\xS-\mathsf{p}}+\fS(\mathsf{p})
\leq\fS(\yS).
\end{equation}
See \cite{Livre1} for an account of Hilbertian nonlinear analysis.

\subsection{Inexact proximal point algorithm \`a la
Br\'ezis--Lions}

The proximal point algorithm plays a central role in nonlinear
analysis \cite[Section~5]{Acnu24}. Its goal is to produce a zero of
a maximally monotone operator. We propose a stochastic extension of
the general version of the deterministic proximal point algorithm
proposed in \cite{Brez78}, together with rates of convergence for
the Yosida approximations which are new even in the deterministic
setting. We recall that $(\forall\upgamma\in\RPP)$ 
$(\forall\xS\in\HS)$ $\xS\in\zer\AS$
$\Leftrightarrow$ $\norm{\moyo{\AS}{\upgamma}\xS}=0$
\cite[Proposition~23.38]{Livre1}. 

\begin{theorem}
\label{t:2}
Let $\AS\colon\HS\to2^{\HS}$ be a maximally monotone operator such
that $\zer\AS\neq\emp$, let $(\upgamma_{\nS})_{\nnn}$ be a sequence
in $\RPP$ such that $\sum_{\nnn}\upgamma_{\nS}^2=\pinf$, 
let $(e_{\nS})_{\nnn}$ be a sequence in 
$L^2(\upOmega,\FE,\PP;\HS)$, and let
$x_0\in L^2(\upOmega,\FE,\PP;\HS)$. Iterate
\begin{equation}
\label{e:p25+}
\begin{array}{l}
\textup{for}\;\nS=0,1,\ldots\\
\left\lfloor
\begin{array}{l}
x_{\nS+1}=\mathsf{J}_{\upgamma_{\nS}\AS}^{}x_{\nS}+e_{\nS}.
\end{array}
\right.\\
\end{array}
\end{equation} 
Set $(\forall\nnn)$ 
$\uptheta_{\nS}=\sum_{\kS=0}^{\nS}\upgamma_{\kS}^2$. 
Then the following hold:
 \begin{enumerate}
\item
\label{t:2i}
Suppose that 
$\sum_{\nnn}\sqrt{\EC{\norm{e_{\nS}}^2}{\FE_{\nS}}}<\pinf\;\Pas$
and $\sum_{\nnn}\upgamma_{\nS+1}^{-2}\uptheta_{\nS}
\EC{\norm{e_{\nS}}^2}{\FE_{\nS}}<\pinf\;\Pas$ Then
$\norm{\moyo{\AS}{\upgamma_{\nS+1}}x_{\nS+1}}
=\mathrm{o}(1/\sqrt{\uptheta_{\nS}})\;\Pas$
\item
\label{t:2ii}
Suppose that $\sum_{\nnn}\sqrt{\EE\norm{e_{\nS}}^2}<\pinf$ and
$\sum_{\nnn}\upgamma_{\nS+1}^{-2}\uptheta_{\nS}
\EE\norm{e_{\nS}}^2<\pinf$. Then
$\norm{\moyo{\AS}{\upgamma_{\nS}}x_{\nS}}
=\mathrm{o}(1/\sqrt{\uptheta_{\nS}})\;\Pas$ and
$\EE\norm{\moyo{\AS}{\upgamma_{\nS}}x_{\nS}}
=\mathrm{o}(1/\sqrt{\uptheta_{\nS}})$.
\end{enumerate}
\end{theorem}
\begin{proof}
Let $\nnn$. It follows from  \eqref{e:p25+}, the monotonicity
of $\AS$, and \cite[Proposition~23.22]{Livre1} that
\begin{align}
0&\leq\upgamma_{\nS+1}^{-1}
\scal1{\mathsf{J}_{\upgamma_{\nS}\AS}x_{\nS}
-\mathsf{J}_{\upgamma_{\nS+1}\AS}x_{\nS+1}}
{\moyo{\AS}{\upgamma_{\nS}}x_{\nS}
-\moyo{\AS}{\upgamma_{\nS+1}}x_{\nS+1}}\nonumber\\
&=\upgamma_{\nS+1}^{-1}
\scal1{x_{\nS+1}-e_{\nS}
-\mathsf{J}_{\upgamma_{\nS+1}\AS}x_{\nS+1}}
{\moyo{\AS}{\upgamma_{\nS}}x_{\nS}
-\moyo{\AS}{\upgamma_{\nS+1}}x_{\nS+1}}\nonumber\\
&=\scal1{\moyo{\AS}
{\upgamma_{\nS+1}}x_{\nS+1}-\upgamma_{\nS+1}^{-1}e_{\nS}}
{\moyo{\AS}{\upgamma_{\nS}}x_{\nS}
-\moyo{\AS}{\upgamma_{\nS+1}}x_{\nS+1}}\nonumber\\
&=\scal1{\moyo{\AS}
{\upgamma_{\nS+1}}x_{\nS+1}}{\moyo{\AS}{\upgamma_{\nS}}x_{\nS}
-\moyo{\AS}{\upgamma_{\nS+1}}x_{\nS+1}}
-\scal1{\upgamma_{\nS+1}^{-1}e_{\nS}}
{\moyo{\AS}{\upgamma_{\nS}}x_{\nS}
-\moyo{\AS}{\upgamma_{\nS+1}}x_{\nS+1}}\nonumber\\
&=\dfrac{1}{2}\brk2{\norm1{\moyo{\AS}{\upgamma_{\nS}}x_{\nS}}^2
-\norm1{\moyo{\AS}{\upgamma_{\nS+1}}x_{\nS+1}}^2
-\norm1{\moyo{\AS}{\upgamma_{\nS}}x_{\nS}
-\moyo{\AS}{\upgamma_{\nS+1}}x_{\nS+1}}^2}
-\scal1{\upgamma_{\nS+1}^{-1}e_{\nS}}
{\moyo{\AS}{\upgamma_{\nS}}x_{\nS}
-\moyo{\AS}{\upgamma_{\nS+1}}x_{\nS+1}}
\nonumber\\
&\leq\dfrac{1}{2}\brk2{\norm1{\moyo{\AS}{\upgamma_{\nS}}x_{\nS}}^2
-\norm1{\moyo{\AS}{\upgamma_{\nS+1}}x_{\nS+1}}^2
-\norm1{\moyo{\AS}{\upgamma_{\nS}}x_{\nS}
-\moyo{\AS}{\upgamma_{\nS+1}}x_{\nS+1}}^2}
\nonumber\\
&\quad\quad
+\dfrac{1}{2}\brk2{
\norm{e_{\nS}}^2+\norm1{\moyo{\AS}{\upgamma_{\nS}}x_{\nS}
-\moyo{\AS}{\upgamma_{\nS+1}}x_{\nS+1}}^2}\;\;\Pas
\end{align}
We take the conditional expectation with respect to $\FE_{\nS}$ 
to get
\begin{equation}
\label{e:pp1}
\EC2{\norm1{\moyo{\AS}{\upgamma_{\nS+1}}x_{\nS+1}}^2}
{\FE_{\nS}}
\leq\norm{\moyo{\AS}{\upgamma_{\nS}}x_{\nS}}^2
+\dfrac{\EC1{\norm{e_{\nS}}^2}{\FE_{\nS}}}{\upgamma_{\nS+1}^2}\;\;
\Pas
\end{equation}
Now, let $\zS\in\zer\AS$. Then $\mathsf{J}_{\upgamma_{\nS}\AS}$ is
firmly nonexpansive \cite[Corollary~23.9]{Livre1} and
$\zS=\mathsf{J}_{\upgamma_{\nS}\AS}\zS$ 
\cite[Proposition~23.38]{Livre1}. Hence, \eqref{e:p25+},
\eqref{e:95}, and the conditional Jensen inequality yield
\begin{align}
\hspace{-3mm}
&\EC1{\norm{x_{\nS+1}-\zS}^2}{\FE_{\nS}}\nonumber\\
&\quad=\norm{\mathsf{J}_{\upgamma_{\nS}\AS}x_{\nS}-\zS}^2
+2\EC1{\scal{\mathsf{J}_{\upgamma_{\nS}\AS}x_{\nS}-\zS}{e_{\nS}}}
{\FE_{\nS}}
+\EC1{\norm{e_{\nS}}^2}{\FE_{\nS}}\nonumber\\
&\quad\leq\norm{x_{\nS}-\zS}^2
-\norm{x_{\nS}-\mathsf{J}_{\upgamma_{\nS}\AS}x_{\nS}}^2
+2\EC1{\scal{\mathsf{J}_{\upgamma_{\nS}\AS}x_{\nS}-\zS}{e_{\nS}}}
{\FE_{\nS}}+\EC1{\norm{e_{\nS}}^2}{\FE_{\nS}}\nonumber\\
&\quad\leq\norm{x_{\nS}-\zS}^2
-\norm{x_{\nS}-\mathsf{J}_{\upgamma_{\nS}\AS}x_{\nS}}^2
+2\norm{\mathsf{J}_{\upgamma_{\nS}\AS}x_{\nS}-\zS}
\EC{\norm{e_{\nS}}}{\FE_{\nS}}
+\EC1{\norm{e_{\nS}}^2}{\FE_{\nS}}\nonumber\\
&\quad\leq\norm{x_{\nS}-\zS}^2
-\upgamma_{\nS}^2\norm{\moyo{\AS}{\upgamma_{\nS}}x_{\nS}}^2
+2\norm{x_{\nS}-\zS}
\sqrt{\EC{\norm{e_{\nS}}^2}{\FE_{\nS}}}
+\EC1{\norm{e_{\nS}}^2}{\FE_{\nS}}\;\Pas
\label{e:96}
\end{align}
Likewise, by the triangle inequality, 
\begin{align}
\label{e:pp8}
\EC1{\norm{x_{\nS+1}-\zS}}{\FE_{\nS}}
&\leq\norm{\mathsf{J}_{\upgamma_{\nS}\AS}x_{\nS}-\zS}
+\EC1{\norm{e_{\nS}}}{\FE_{\nS}}\nonumber\\
&\leq\norm{x_{\nS}-\zS}+\EC1{\norm{e_{\nS}}}{\FE_{\nS}}
\nonumber\\
&\leq\norm{x_{\nS}-\zS}+\sqrt{\EC1{\norm{e_{\nS}}^2}{\FE_{\nS}}}
\;\Pas
\end{align}
Furthermore, taking the expected value of \eqref{e:96} and
using the Cauchy--Schwarz inequality yields
\begin{align}
\label{e:pp5}
\EE{\norm{x_{\nS+1}-\zS}^2}
+\upgamma_{\nS}^2\EE\norm{\moyo{\AS}{\upgamma_{\nS}}x_{\nS}}^2
&\leq\EE\norm{x_{\nS}-\zS}^2
+2\EE\brk3{\norm{x_{\nS}-\zS}
\sqrt{\EC1{\norm{e_{\nS}}^2}{\FE_{\nS}}}}
+\EE{\norm{e_{\nS}}^2}\nonumber\\
&\leq\EE\norm{x_{\nS}-\zS}^2
+2\sqrt{\EE{\norm{x_{\nS}-\zS}^2}}
\sqrt{\EE{\norm{e_{\nS}}^2}}+\EE{\norm{e_{\nS}}^2},
\end{align}
while using Minkowski's inequality yields
\begin{equation}
\label{e:pp6}
\sqrt{\EE\norm{x_{\nS+1}-\zS}^2}
\leq\sqrt{\EE\norm1{\mathsf{J}_{\upgamma_{\nS}\AS}x_{\nS}-\zS}^2}+
\sqrt{\EE\norm{e_{\nS}}^2}
\leq\sqrt{\EE\norm{x_{\nS}-\zS}^2}+
\sqrt{\EE\norm{e_{\nS}}^2}.
\end{equation}
Finally, let us set
\begin{equation}
\label{e:pp9}
(\forall\nnn)\quad
\begin{cases}
\alpha_{\nS}=\upgamma_{\nS}^2\\
\xi_{\nS}=\norm{\moyo{\AS}{\upgamma_{\nS}}x_{\nS}}^2\\
\delta_{\nS}=0\\
\varepsilon_{\nS}
=\dfrac{\EC1{\norm{e_{\nS}}^2}{\FE_{\nS}}}{\upgamma_{\nS+1}^2}.
\end{cases}
\end{equation}
Then $\sum_{\nnn}\alpha_{\nS}=\pinf\;\Pas$ and, in view of
\eqref{e:pp1}, the almost supermartingale property \eqref{e:c1}
holds.

\ref{t:2i}: 
We infer from \eqref{e:pp9} that 
\begin{equation}
\label{e:90}
\sum_{\nnn}\uptheta_{\nS}\varepsilon_{\nS}
=\sum_{\nnn}\dfrac{\uptheta_{\nS}\EC{\norm{e_{\nS}}^2}
{\FE_{\nS}}}{\upgamma_{\nS+1}^2}<\pinf\;\Pas
\end{equation}
Next, let $\zS\in\zer\AS$. We derive from \eqref{e:pp8} and
Lemma~\ref{l:5} that $(\norm{x_{\nS}-\zS})_{\nnn}$ converges $\Pas$
As a result, $(\norm{x_{\nS}-\zS})_{\nnn}$ is bounded $\Pas$ and
hence
\begin{equation}
\label{e:91}
\sum_{\nnn}\brk2{\norm{x_{\nS}-\zS}
\sqrt{\EC{\norm{e_{\nS}}^2}{\FE_{\nS}}}
+\EC1{\norm{e_{\nS}}^2}{\FE_{\nS}}}<\pinf\;\;\Pas
\end{equation}
It therefore follows from \eqref{e:96} and Lemma~\ref{l:5} applied
to
\begin{equation}
(\forall\nnn)\quad
\begin{cases}
\beta_{\nS}=\norm{x_{\nS}-\zS}^2\\
\delta_{\nS}=0\\
\eta_{\nS}=\upgamma_{\nS}^2
\norm{\moyo{\AS}{\upgamma_{\nS}}x_{\nS}}^2\\
\tau_{\nS}=2\norm{x_{\nS}-\zS}
\sqrt{\EC{\norm{e_{\nS}}^2}{\FE_{\nS}}}
+\EC1{\norm{e_{\nS}}^2}{\FE_{\nS}}
\end{cases}
\end{equation}
that 
\begin{equation}
\label{e:pp2}
\sum_{\nnn}\alpha_{\nS}\xi_{\nS}
=\sum_{\nnn}\upgamma_{\nS}^2
\norm{\moyo{\AS}{\upgamma_{\nS}}x_{\nS}}^2
=\sum_{\nnn}\eta_{\nS}<\pinf\;\;\Pas
\end{equation}
Altogether, we derive from Theorem~\ref{t:1}\ref{t:1i} applied to
\eqref{e:pp9} that $\norm{\moyo{\AS}{\upgamma_{\nS+1}}x_{\nS+1}}
=\mathrm{o}(1/\sqrt{\uptheta_{\nS}})\;\Pas$

\ref{t:2ii}: 
Let $\zS\in\zer\AS$. We derive from \eqref{e:pp6} and
Corollary~\ref{c:3} that $(\EE\norm{x_{\nS}-\zS}^2)_{\nnn}$ is
bounded. Consequently, 
\begin{equation}
\sum_{\nnn}\brk3{\sqrt{\EE{\norm{x_{\nS}-\zS}^2}}
\sqrt{\EE{\norm{e_{\nS}}^2}}+\EE{\norm{e_{\nS}}^2}}<\pinf.
\end{equation}
Using the notation of \eqref{e:pp9}, it then
results from \eqref{e:pp5} and Corollary~\ref{c:3} that
\begin{equation}
\sum_{\nnn}\EE\brk{\alpha_{\nS}\xi_{\nS}}=\sum_{\nnn}
\upgamma_{\nS}^2\EE
\brk1{\norm{\moyo{\AS}{\upgamma_{\nS}}x_{\nS}}^2}<\pinf.
\end{equation}
On the other hand,
$\sum_{\nnn}\EE\brk{\uptheta_{\nS}\varepsilon_{\nS}}=
\sum_{\nnn}\upgamma_{\nS+1}^{-2}\uptheta_{\nS}
\EE\norm{e_{\nS}}^2<\pinf$.
Altogether, Theorem~\ref{t:1}\ref{t:1iii} applied to \eqref{e:pp9}
yields $\norm{\moyo{\AS}{\upgamma_{\nS}}x_{\nS}}
=\mathrm{o}(1/\sqrt{\uptheta_{\nS}})\;\Pas$ and
$\EE\norm{\moyo{\AS}{\upgamma_{\nS}}x_{\nS}}
=\mathrm{o}(1/\sqrt{\uptheta_{\nS}})$.
\end{proof}

\begin{remark}
\label{r:3}
Let $\fS\in\upGamma_0(\HS)$ and let 
$\moyo{\fS}{\upgamma}\colon\xS\mapsto\min_{\yS\in\HS}
\brk{\fS(\yS)+\norm{\xS-\yS}^2/2}$ denote its Moreau envelope of
index $\upgamma\in\RPP$. Then $\AS$ is maximally monotone,
$\mathsf{J}_{\AS}=\prox_{\fS}$,
and $(\forall\xS\in\HS)$ 
$\xS\in\Argmin\fS$ $\Leftrightarrow$ 
$\xS\in\zer\uppartial\fS$ $\Leftrightarrow$ 
$\norm{\nabla\moyo{\fS}{\upgamma}(\xS)}=0$ and
$\moyo{\brk{\uppartial\fS}}{\upgamma}(\xS)=
\nabla\moyo{\fS}{\upgamma}(\xS)$ \cite{Livre1}. 
Thus, if $\AS=\uppartial\fS$, \eqref{e:p25+} becomes
\begin{equation}
\label{e:r25+}
\begin{array}{l}
\textup{for}\;\nS=0,1,\ldots\\
\left\lfloor
\begin{array}{l}
x_{\nS+1}=\prox_{\upgamma_{\nS}\fS}x_{\nS}+e_{\nS},
\end{array}
\right.\\
\end{array}
\end{equation}
and the conclusions of Theorem~\ref{t:2} are:
\begin{enumerate}
\item
$\norm{\nabla\moyo{\fS}{\upgamma_{\nS+1}}(x_{\nS+1})}
=\mathrm{o}(1/\sqrt{\uptheta_{\nS}})\;\Pas$
\item
$\norm{\nabla\moyo{\fS}{\upgamma_{\nS}}(x_{\nS})}
=\mathrm{o}(1/\sqrt{\uptheta_{\nS}})\;\Pas$ and 
$\EE\norm{\nabla\moyo{\fS}{\upgamma_{\nS}}(x_{\nS})}
=\mathrm{o}(1/\sqrt{\uptheta_{\nS}})$.
\end{enumerate}
\end{remark}

\begin{remark}
Suppose that 
$\sum_{\nnn}\sqrt{\EE\norm{e_{\nS}}^2}<\pinf$. Then the sequence
$(x_{\nS})_{\nnn}$ generated by \eqref{e:p25+} converges weakly
$\Pas$ and weakly in $L^2(\upOmega,\FE,\PP;\HS)$ to a
$(\zer\AS)$-valued random variable \cite[Theorem~4.1(iii)]{Sadd25}.
\end{remark}

\subsection{Inexact proximal point algorithm for minimization}

We saw in Remark~\ref{r:3} that the proximal point algorithm
\eqref{e:p25+} can be used to minimize convex functions. However,
as already noted in the deterministic setting of \cite{Brez78},
sharper results can be obtained by exploiting more finely the
specific properties of subdifferentials. Thus, the condition
$\sum_{\nnn}\upgamma_{\nS}^2=\pinf$ was relaxed to
$\sum_{\nnn}\upgamma_{\nS}=\pinf$ in \cite{Brez78}. With this in
mind, we propose a new version of the stochastic proximal point
algorithm in which we introduce the following stochastic relaxation
of \eqref{e:94}.

\begin{definition}
\label{d:2}
Let $\fS\in\upGamma_0(\HS)$, let $\mu$ and $\rho$ be
$\RP$-valued random variables, let $\EuScript{A}$ be a
sub-$\upsigma$-algebra of $\FE$, let 
$x\in L^2(\upOmega,\EuScript{A},\PP;\HS)$, and let
$p\in L^2(\upOmega,\FE,\PP;\HS)$. 
Then
\begin{multline}
p\underset{\mu,\rho}{\overset{\EuScript{A}}{\approx}}
\prox_{\fS}x\;\Leftrightarrow\;
\brk1{\forall y\in L^2(\upOmega,\EuScript{A},\PP;\HS)}\quad
\EC1{\scal{y-p}{x-p}+\fS(p)}{\EuScript{A}}\\
\leq\brk1{1+\EC{\mu}{\EuScript{A}}}\fS(y)
+\EC{\rho}{\EuScript{A}}\;\Pas 
\end{multline}
\end{definition}

\begin{theorem}
\label{t:5}
Let $\fS\in\upGamma_0(\HS)$ be such that $\Argmin\fS\neq\emp$, let
$(\upgamma_{\nS})_{\nnn}$ be a decreasing sequence in $\RPP$ 
such that $\sum_{\nnn}\upgamma_{\nS}=\pinf$, let
$(\mu_{\nS})_{\nnn}$ and $(\rho_{\nS})_{\nnn}$ be sequences of 
$\RP$-valued random variables, and let $x_0\in
L^2(\upOmega,\FE,\PP;\HS)$. Iterate
\begin{equation}
\label{e:30}
\begin{array}{l}
\textup{for}\;\nS=0,1,\ldots\\
\left\lfloor
\begin{array}{l}
x_{\nS+1}\underset{\mu_{\nS},\rho_{\nS}}
{\overset{\FE_{\nS}}{\approx}}
\prox_{\upgamma_{\nS}\fS}x_{\nS}
\end{array}
\right.\\
\end{array}
\end{equation}
and suppose that 
$\sum_{\nnn}\EC{\mu_{\nS}}{\FE_{\nS}}<\pinf\;\Pas$ and 
$\sum_{\nnn}\upgamma_{\nS}^{-1}\EC{\rho_{\nS}}{\FE_{\nS}}
<\pinf\;\Pas$
Then the following hold:
\begin{enumerate}
\item
\label{t:5i}
$(x_{\nS})_{\nnn}$ converges weakly $\Pas$ 
to an $(\Argmin\fS)$-valued random variable.
\item
\label{t:5ii}
Suppose that $\sum_{\nnn}\EE\mu_{\nS}<\pinf$ and 
$\sum_{\nnn}\EE\rho_{\nS}<\pinf$. Then
$(x_{\nS})_{\nnn}$ converges weakly in $L^2(\upOmega,\FE,\PP;\HS)$
to an $(\Argmin\fS)$-valued random variable.
\item
\label{t:5iii}
Set $(\forall\nnn)$ 
$\uptheta_{\nS}=\sum_{\kS=0}^{\nS}\upgamma_{\kS}$. Suppose that 
$\sum_{\nnn}\uptheta_{\nS}\EC1{\mu_{\nS}}{\FE_{\nS}}<\pinf\;\Pas$ 
and that $\sum_{\nnn}\upgamma_{\nS}^{-1}\uptheta_{\nS}
\EC1{\rho_{\nS}}{\FE_{\nS}}<\pinf\;\Pas$
Then $\fS(x_{\nS})-\inf\fS(\HS)
=\mathrm{o}(1/\uptheta_{\nS})\;\Pas$
\end{enumerate}
\end{theorem}
\begin{proof}
We first observe that
\begin{equation}
\label{e:72}
\sum_{\nnn}\EC{\rho_{\nS}}{\FE_{\nS}}\leq\upgamma_0
\sum_{\nnn}\upgamma_{\nS}^{-1}\EC{\rho_{\nS}}{\FE_{\nS}}
<\pinf\;\Pas
\end{equation}
Now let $\nS\in\NN$. Then \eqref{e:30} and Definition~\ref{d:2}
yield
\begin{multline}
\label{e:70}
\brk1{\forall y\in L^2(\upOmega,\FE_{\nS},\PP;\HS)}\quad
\EC{\scal{y-x_{\nS+1}}{x_{\nS}-x_{\nS+1}}+
\upgamma_{\nS}\fS(x_{\nS+1})}{\FE_{\nS}}\\
\leq\upgamma_{\nS}\brk1{1+\EC{\mu_{\nS}}{\FE_{\nS}}}
\fS(y)+\EC{\rho_{\nS}}{\FE_{\nS}}\;\;\Pas
\end{multline}
In particular, for $y=x_{\nS}$, we obtain
\begin{align}
\label{e:73}
\EC1{\fS(x_{\nS+1})}{\FE_{\nS}}
&\leq\EC1{\upgamma_{\nS}^{-1}\norm{x_{\nS}-x_{\nS+1}}^2
+\fS(x_{\nS+1})}{\FE_{\nS}}
\nonumber\\
&\leq\brk1{1+\EC{\mu_{\nS}}{\FE_{\nS}}}
\fS(x_{\nS})+\upgamma_{\nS}^{-1}\EC{\rho_{\nS}}{\FE_{\nS}}\;\;\Pas
\end{align}
Now let $\zS\in\Argmin\fS$ and set
\begin{equation}
\label{e:78}
\xi_{\nS}=\fS(x_{\nS})-\inf\fS(\HS).
\end{equation}
Since
\begin{align}
\norm{x_{\nS+1}-\zS}^2
&=\norm{x_{\nS}-\zS}^2
+2\scal{x_{\nS}-\zS}{x_{\nS+1}-x_{\nS}}
+\norm{x_{\nS+1}-x_{\nS}}^2
\nonumber\\
&=\norm{x_{\nS}-\zS}^2
+2\scal{x_{\nS+1}-\zS}{x_{\nS+1}-x_{\nS}}
-\norm{x_{\nS+1}-x_{\nS}}^2
\nonumber\\
&\leq\norm{x_{\nS}-\zS}^2
+2\scal{\zS-x_{\nS+1}}{x_{\nS}-x_{\nS+1}}\;\;\Pas,
\end{align}
we deduce from \eqref{e:70} that
\begin{align}
\label{e:77}
&\EC1{\norm{x_{\nS+1}-\zS}^2}{\FE_\nS}\nonumber\\
&\quad\leq\norm{x_{\nS}-\zS}^2
+2\EC1{\scal{\zS-x_{\nS+1}}{x_{\nS}-x_{\nS+1}}}{\FE_\nS}
\nonumber\\
&\quad\leq\norm{x_{\nS}-\zS}^2
-2\upgamma_{\nS}\EC1{\xi_{\nS+1}}{\FE_{\nS}}
+2\upgamma_{\nS}\brk1{\inf\fS(\HS)}\EC{\mu_{\nS}}{\FE_{\nS}}
+2\EC{\rho_{\nS}}{\FE_{\nS}}\nonumber\\
&\quad\leq\norm{x_{\nS}-\zS}^2
-2\upgamma_{\nS}\EC1{\xi_{\nS+1}}{\FE_{\nS}}
+2\upgamma_{\nS}\abs{\inf\fS(\HS)}\,\EC{\mu_{\nS}}{\FE_{\nS}}
+2\EC{\rho_{\nS}}{\FE_{\nS}}\nonumber\\
&\quad\leq\norm{x_{\nS}-\zS}^2
-2\upgamma_{\nS+1}\EC1{\xi_{\nS+1}}{\FE_{\nS}}
+2\upgamma_{0}\abs{\inf\fS(\HS)}\,\EC{\mu_{\nS}}{\FE_{\nS}}
+2\EC{\rho_{\nS}}{\FE_{\nS}}\;\;\Pas
\end{align}
and therefore from \eqref{e:72} and Lemma~\ref{l:5} that
\begin{equation}
\label{e:74}
\sum_{\nnn}\upgamma_{\nS+1}\EC1{\xi_{\nS+1}}{\FE_{\nS}}<\pinf\;
\Pas
\end{equation}

\ref{t:5i}: 
Since $\fS$ is lower semicontinuous, we derive from \eqref{e:78}
that $(\upgamma_{\nS}\xi_{\nS})_{\nnn}\in\elll(\mathfrak{F};\RP)$.
Consequently, it follows from \eqref{e:74} and Corollary~\ref{c:4}
that 
\begin{equation}
\label{e:60}
\sum_{\nnn}\upgamma_{\nS}\brk1{\fS(x_{\nS})-\inf\fS(\HS)}
<\pinf\;\Pas
\end{equation}
On the other hand, because $\sum_{\nnn}\upgamma_{\nS}=\pinf$, there
exists a strictly increasing sequence $(\kS_{\nS})_{\nnn}$ in $\NN$
such that 
\begin{equation}
\label{e:75}
\fS(x_{\kS_{\nS}})\to\inf\fS(\HS)\;\Pas
\end{equation}
At the same time, \eqref{e:73} and Lemma~\ref{l:5} imply that 
\begin{equation}
\brk1{\fS(x_{\nS})}_{\nnn}\;\text{converges}\;\Pas
\end{equation}
Altogether, 
\begin{equation}
\label{e:79}
\fS(x_{\nS})\to\inf\fS(\HS)\;\Pas
\end{equation}
Now let $x$ be a \Pas weak sequential cluster point of 
$(x_{\nS})_{\nnn}$, say $x_{\kS_{\nS}}\weakly x\;\Pas$
Then, by weak lower semicontinuity of $\fS$,
$\fS(x)\leq\varliminf\fS(x_{\kS_{\nS}})=\inf\fS(\HS)$. This
confirms that every weak sequential cluster point of
$(x_{\nS})_{\nnn}$ is in $\Argmin\fS\;\Pas$ By invoking
\eqref{e:77}, \eqref{e:72}, and \cite[Proposition~2.3(iv)]{Siop15}
(see also \cite[Proof of Lemma~2.11(iii)]{Moco26}), we conclude
that $x_{\nS}\weakly x\in\Argmin\fS\;\Pas$ 

\ref{t:5ii}: We have
\begin{equation}
\sum_{\nnn}\EE\brk2{\upgamma_{0}\abs{\inf\fS(\HS)}\,
\EC{\mu_{\nS}}{\FE_{\nS}}
+\EC{\rho_{\nS}}{\FE_{\nS}}}
=\upgamma_{0}\abs{\inf\fS(\HS)}\,\sum_{\nnn}
\EE\mu_{\nS}+\sum_{\nnn}\EE\rho_{\nS}<\pinf.
\end{equation}
Hence, the assertion
follows from \eqref{e:77} and \ref{t:5i} by arguing as in the 
proof of \cite[Theorem~3.2(vi)(d)]{Moco26}. 

\ref{t:5iii}: We derive from \eqref{e:73} and \eqref{e:78} that
\begin{align}
\label{e:71}
\EC1{\xi_{\nS+1}}{\FE_{\nS}}
&\leq\brk1{1+\EC{\mu_{\nS}}{\FE_{\nS}}}
\xi_{\nS}+\brk1{\inf\fS(\HS)}\EC{\mu_{\nS}}{\FE_{\nS}}
+\upgamma_{\nS}^{-1}\EC{\rho_{\nS}}{\FE_{\nS}}
\nonumber\\
&\leq\brk1{1+\EC{\mu_{\nS}}{\FE_{\nS}}}
\xi_{\nS}+\abs{\inf\fS(\HS)}\,\EC{\mu_{\nS}}{\FE_{\nS}}
+\upgamma_{\nS}^{-1}\EC{\rho_{\nS}}{\FE_{\nS}}\;\;\Pas
\end{align}
We apply Theorem~\ref{t:1}\ref{t:1ii} with \eqref{e:78} and  
\begin{equation}
(\forall\nnn)\quad
\begin{cases}
\alpha_{\nS}=\upgamma_{\nS}\\
\delta_{\nS}=\EC{\mu_{\nS}}{\FE_{\nS}}\\
\varepsilon_{\nS}=\abs{\inf\fS(\HS)}\,\EC{\mu_{\nS}}{\FE_{\nS}}+
\upgamma_{\nS}^{-1}\EC{\rho_{\nS}}{\FE_{\nS}}
\end{cases}
\end{equation}
to obtain the desired conclusion thanks to \eqref{e:74}. 
\end{proof}

\subsection{Randomly relaxed inexact Krasnosel'ski\u\i--Mann 
iterations}

The Krasnosel'ski\u\i--Mann iterative process aims at constructing
a fixed point of a nonexpansive operator $\TS\colon\HS\to\HS$ and
it is at the core of a number of nonlinear algorithms
\cite{Livre1,Acnu24,Dong22}. In its deterministic version, the
algorithm is governed by the update \cite{Groe72}
\begin{equation}
\label{e:km5}
(\forall\nnn)\quad
\xS_{\nS+1}=\xS_{\nS}+\upmu_{\nS}\brk1{\TS\xS_{\nS}-\xS_{\nS}},
\end{equation}
where $(\upmu_{\nS})_{\nnn}$ is a sequence of relaxation parameters
in $\zeroun$ satisfying 
$\sum_{\nnn}\upmu_{\nS}(1-\upmu_{\nS})=\pinf$. When the
implementation is subject to stochastic approximations,
\eqref{e:km5} becomes
\begin{equation}
\label{e:km51}
(\forall\nnn)\quad
x_{\nS+1}=x_{\nS}+\upmu_{\nS}\brk1{\TS x_{\nS}+e_{\nS}-x_{\nS}},
\end{equation}
and almost sure weak convergence of $(x_{\nS})_{\nnn}$ was
demonstrated in \cite{Siop15} (see also
\cite{Brav24} for the finite-dimensional setting). Extensions of
these convergence results to the case when random relaxations 
$(\mu_{\nS})_{\nnn}$ are employed in \eqref{e:km51} were
established in \cite{Moco26}. Thus far, however, rates on the
displacement sequence $(\norm{\TS x_{\nS}-x_{\nS}})_{\nnn}$ have
not been derived in the stochastic setting. We address this gap in
the following theorem, which features not only stochastic
approximations but also random relaxations.

\begin{theorem}
\label{t:4}
Let $\TS\colon\HS\to\HS$ be a nonexpansive operator such that
$\Fix\TS\neq\emp$ and let 
$x_{0}\in L^2(\upOmega,\FE,\PP;\HS)$. Iterate
\begin{equation}
\label{e:p25}
\begin{array}{l}
\textup{for}\;\nS=0,1,\ldots\\
\left\lfloor
\begin{array}{l}
e_{\nS}\in L^2(\upOmega,\FE,\PP;\HS)\\
\mu_{\nS}\in
L^\infty(\upOmega,\FE,\PP;\left]0,1\right[)\;
\textup{is independent of}\;(x_0,\ldots,x_{\nS},e_{\nS})\\
x_{\nS+1}=x_{\nS}+\mu_{\nS}\brk1{\TS x_{\nS}+e_{\nS}-x_{\nS}}.
\end{array}
\right.\\
\end{array}
\end{equation}
Set $(\forall\nnn)$ 
$\uptheta_{\nS}=\sum_{\kS=0}^{\nS}\EE\brk{\mu_{\kS}(1-\mu_{\kS})}$.
Suppose that 
\begin{equation}
\label{e:km1}
\begin{cases}
\sum_{\nnn}\EE\brk1{\mu_{\nS}(1-\mu_{\nS})}=\pinf\\
\sum_{\nnn}\EE\mu_{\nS}\sqrt{\EC1{\norm{e_{\nS}}^2}
{\FE_{\nS}}}<\pinf\;\Pas\\[2mm]
\sum_{\nnn}\uptheta_{\nS}\EE\brk2{\dfrac{\mu_{\nS}}{1-\mu_{\nS}}}
\EC1{\norm{e_{\nS}}^2}{\FE_{\nS}}<\pinf\;\Pas
\end{cases}
\end{equation}
Then $\norm{\TS x_{\nS+1}-x_{\nS+1}}
=\mathrm{o}(1/\sqrt{\uptheta_{\nS}})\;\Pas$
\end{theorem}
\begin{proof}
Let $\zS\in\Fix\TS$. We derive from \eqref{e:p25} and the
nonexpansiveness of $\TS$ that 
\begin{align}
(\forall\nnn)\quad
\norm{x_{\nS+1}-\zS}
&\leq\brk{1-\mu_{\nS}}\norm{x_{\nS}-\zS}+
\mu_{\nS}\norm{\TS x_{\nS}-\zS}+\mu_{\nS}\norm{e_{\nS}}
\nonumber\\
&\leq\norm{x_{\nS}-\zS}+\mu_{\nS}\norm{e_{\nS}}\;\Pas
\end{align}
Hence, by Lemma~\ref{l:6},
\begin{align}
(\forall\nnn)\quad
\EC1{\norm{x_{\nS+1}-\zS}}{\FE_{\nS}}
&\leq\norm{x_{\nS}-\zS}+\EE\mu_{\nS}\EC1{\norm{e_{\nS}}}{\FE_{\nS}}
\nonumber\\
&\leq\norm{x_{\nS}-\zS}+\EE\mu_{\nS}
\sqrt{\EC1{\norm{e_{\nS}}^2}{\FE_{\nS}}}\;\;\Pas
\end{align}
It then follows from Lemma~\ref{l:5} that 
$(\norm{x_{\nS}-\zS})_{\nnn}$
converges $\Pas$ and hence that it is bounded $\Pas$ 
On the other hand, we deduce from \eqref{e:p25} and
\cite[Corollary~2.15]{Livre1} that
\begin{align}
\hspace{-5mm}
(\forall\nnn)\quad
\norm{x_{\nS+1}-\zS}^2
&=\brk{1-\mu_{\nS}}\norm{x_{\nS}-\zS}^2+
\mu_{\nS}\norm{\TS x_{\nS}-\zS+e_{\nS}}^2
-\mu_{\nS}(1-\mu_{\nS})\norm{\TS x_{\nS}-x_{\nS}+e_{\nS}}^2
\nonumber\\
&=\brk{1-\mu_{\nS}}\norm{x_{\nS}-\zS}^2+\mu_{\nS}\brk1{
\norm{\TS x_{\nS}-\zS}^2
+2\scal{\TS x_{\nS}-\zS}{e_{\nS}}
+\norm{e_{\nS}}^2}\nonumber\\
&\quad-\mu_{\nS}(1-\mu_{\nS})\brk1{
\norm{\TS x_{\nS}-x_{\nS}}^2
+2\scal{\TS x_{\nS}-x_{\nS}}{e_{\nS}}
+\norm{e_{\nS}}^2}
\nonumber\\
&\leq\norm{x_{\nS}-\zS}^2-
\mu_{\nS}(1-\mu_{\nS})\norm{\TS x_{\nS}-x_{\nS}}^2
+\mu_{\nS}\brk1{
2\norm{\TS x_{\nS}-\zS}\norm{e_{\nS}}
+\norm{e_{\nS}}^2}\nonumber\\
&\quad+2\mu_{\nS}(1-\mu_{\nS})
\norm{\TS x_{\nS}-x_{\nS}}\norm{e_{\nS}}
-\mu_{\nS}(1-\mu_{\nS})\norm{e_{\nS}}^2
\nonumber\\
&=\norm{x_{\nS}-\zS}^2-
\mu_{\nS}(1-\mu_{\nS})\norm{\TS x_{\nS}-x_{\nS}}^2
+2\mu_{\nS}\norm{\TS x_{\nS}-\zS}\norm{e_{\nS}}
+\mu_{\nS}^2\norm{e_{\nS}}^2\nonumber\\
&\quad+2\mu_{\nS}(1-\mu_{\nS})
\norm{\TS x_{\nS}-x_{\nS}}\norm{e_{\nS}}\;\Pas
\end{align}
Hence, Lemma~\ref{l:6} and the conditional Jensen inequality yield
\begin{align}
\label{e:km6}
&\hspace{-7mm}
(\forall\nnn)\quad\EC1{\norm{x_{\nS+1}-\zS}^2}{\FE_{\nS}}
\nonumber\\
&\leq\norm{x_{\nS}-\zS}^2
-\EE\brk1{\mu_{\nS}(1-\mu_{\nS})}\norm{\TS x_{\nS}-x_{\nS}}^2
+2\EE\brk{\mu_{\nS}}\norm{\TS x_{\nS}-\zS}
\EC1{\norm{e_{\nS}}}{\FE_{\nS}}
\nonumber\\
&\quad
+\EE{\mu_{\nS}^2}\EC1{\norm{e_{\nS}}^2}{\FE_{\nS}}
+2\EE\brk1{\mu_{\nS}(1-\mu_{\nS})}\norm{\TS x_{\nS}-x_{\nS}}
\EC1{\norm{e_{\nS}}}{\FE_{\nS}}
\nonumber\\
&\leq\norm{x_{\nS}-\zS}^2
-\EE\brk1{\mu_{\nS}(1-\mu_{\nS})}\norm{\TS x_{\nS}-x_{\nS}}^2
+2\norm{\TS x_{\nS}-\zS}\EE\mu_{\nS}
\sqrt{\EC1{\norm{e_{\nS}}^2}{\FE_{\nS}}}
\nonumber\\
&\quad
+\EE{\mu_{\nS}^2}\EC1{\norm{e_{\nS}}^2}{\FE_{\nS}}
+2\norm{\TS x_{\nS}-x_{\nS}}\EE\brk1{\mu_{\nS}(1-\mu_{\nS})}
\sqrt{\EC1{\norm{e_{\nS}}^2}{\FE_{\nS}}}
\;\;\Pas
\end{align}
The boundedness of $(\norm{x_{\nS}-\zS})_{\nnn}$ and the
nonexpansiveness of $\TS$ guarantee that 
$(\norm{\TS x_{\nS}-\zS})_{\nnn}$ and 
$(\norm{\TS x_{\nS}-x_{\nS}})_{\nnn}$ are
bounded $\Pas$ Note that, for $\nS$ large enough,
$\uptheta_{\nS}\geq 1$ and 
$\mu_{\nS}^2\leq\mu_{\nS}\leq\mu_{\nS}/(1-\mu_{\nS})$.
Thus, we deduce from \eqref{e:km1} that
\begin{multline}
\sum_{\nnn}\Biggl(
2\norm{\TS x_{\nS}-\zS}\EE\mu_{\nS}
\sqrt{\EC1{\norm{e_{\nS}}^2}{\FE_{\nS}}}
+\EE\mu_{\nS}^2\EC1{\norm{e_{\nS}}^2}{\FE_{\nS}}\\
+2\norm{\TS x_{\nS}-x_{\nS}}\EE\brk1{\mu_{\nS}(1-\mu_{\nS})}
\sqrt{\EC1{\norm{e_{\nS}}^2}{\FE_{\nS}}}\Biggr)
<\pinf\;\;\Pas
\end{multline}
Therefore, we apply Lemma~\ref{l:5} to \eqref{e:km6} to obtain 
\begin{equation}
\label{e:km7}
\sum_{\nnn}\EE\brk1{\mu_{\nS}(1-\mu_{\nS})}
\norm{\TS x_{\nS}-x_{\nS}}^2<\pinf\;\;\Pas
\end{equation}
Now, for every $\nnn$, set $d_{\nS}=\TS x_{\nS+1}-\TS
x_{\nS}-\brk{x_{\nS+1}-x_{\nS}}$ and note that, by nonexpansiveness
of $\TS$ and Lemma~\ref{l:6}, 
\begin{align}
&\EC1{\norm{\TS x_{\nS+1}-x_{\nS+1}}^2}{\FE_{\nS}}
\nonumber\\
&\quad
=\norm{\TS x_{\nS}-x_{\nS}}^2
+2\EC1{\scal{\TS x_{\nS}-x_{\nS}}{d_{\nS}}}{\FE_{\nS}}
+\EC1{\norm{d_{\nS}}^2}{\FE_{\nS}}
\nonumber\\
&\quad
=\norm{\TS x_{\nS}-x_{\nS}}^2
+\EC2{\dfrac{2}{\mu_{\nS}}
\scal{x_{\nS+1}-x_{\nS}-\mu_{\nS}e_{\nS}}{d_{\nS}}}{\FE_{\nS}}
+\EC1{\norm{d_{\nS}}^2}{\FE_{\nS}}
\nonumber\\
&\quad
=\norm{\TS x_{\nS}-x_{\nS}}^2+\EC1{\norm{d_{\nS}}^2}{\FE_{\nS}}
\nonumber\\
&\qquad
+\EC2{\dfrac{1}{\mu_{\nS}}
\brk1{\norm{\TS x_{\nS+1}-\TS x_{\nS}}^2
-\norm{x_{\nS+1}-x_{\nS}}^2-\norm{d_{\nS}}^2}
-2\scal{e_{\nS}}{d_{\nS}}}{\FE_{\nS}}
\nonumber\\
&\quad
\leq\norm{\TS x_{\nS}-x_{\nS}}^2+\EC1{\norm{d_{\nS}}^2}{\FE_{\nS}}
-\EC2{\dfrac{1}{\mu_{\nS}}{\norm{d_{\nS}}^2}
+2\scal{e_{\nS}}{d_{\nS}}}{\FE_{\nS}}
\nonumber\\
&\quad
=\norm{\TS x_{\nS}-x_{\nS}}^2-\EC2{\dfrac{1-\mu_{\nS}}{\mu_{\nS}}
\norm{d_{\nS}}^2}{\FE_{\nS}}
-2\EC1{\scal{e_{\nS}}{d_{\nS}}}{\FE_{\nS}}
\nonumber\\
&\quad
=\norm{\TS x_{\nS}-x_{\nS}}^2-\EC3{\dfrac{1-\mu_{\nS}}{\mu_{\nS}}
\norm2{d_{\nS}+\dfrac{\mu_{\nS}}{1-\mu_{\nS}}e_{\nS}}^2}{\FE_{\nS}}
+\EC2{\dfrac{\mu_{\nS}}{1-\mu_{\nS}}\norm{e_{\nS}}^2}{\FE_{\nS}}
\nonumber\\
&\quad
\leq
\norm{\TS x_{\nS}-x_{\nS}}^2
+\EC2{\dfrac{\mu_{\nS}}{1-\mu_{\nS}}\norm{e_{\nS}}^2}{\FE_{\nS}}
\nonumber\\
&\quad=\norm{\TS x_{\nS}-x_{\nS}}^2
+\EE\brk2{\dfrac{\mu_{\nS}}{1-\mu_{\nS}}}
\EC1{\norm{e_{\nS}}^2}{\FE_{\nS}}\;\;\Pas
\label{e:km8}
\end{align}
In view of \eqref{e:km1}, \eqref{e:km7}, and \eqref{e:km8}, upon
applying Theorem~\ref{t:1}\ref{t:1i} with
\begin{equation}
(\forall\nnn)\quad
\begin{cases}
\alpha_{\nS}=\EE\brk1{\mu_{\nS}(1-\mu_{\nS})}\\
\xi_{\nS}=\norm{\TS x_{\nS}-x_{\nS}}^2\\
\delta_{\nS}=0\\
\varepsilon_{\nS}=\EE\brk3{\dfrac{\mu_{\nS}}{1-\mu_{\nS}}}
\EC1{\norm{e_{\nS}}^2}{\FE_{\nS}},
\end{cases}
\end{equation}
we conclude that $\norm{\TS x_{\nS+1}-x_{\nS+1}}
=\mathrm{o}(1/\sqrt{\uptheta_{\nS}})\;\Pas$ 
\end{proof}

\subsection{Stochastic gradient method}
\label{sec:31}

Given a family $(\fS_{\kS})_{\kS\in\mathsf{K}}$ of convex smooth
functions on $\HS$ and a measurable map
$k\colon\upOmega\to\mathsf{K}$, we consider the problem of
minimizing $\fS\colon\xS\mapsto
\int_{\upOmega}\fS_{k(\upomega)}(\xS)\PP(d\upomega)$.
The principle of the stochastic gradient method is to seek a
minimizer by performing at each iteration a gradient step with
respect to a randomly selected function $\fS_{\mathsf{k}}$.
While rates of convergence for the objective function gap 
$(\fS(x_{\nS})-\inf\fS(\HS))_{\nnn}$ have been extensively studied,
existing analyses typically rely on stronger notions of
convexity \cite{Liu24}, establish rates only for ergodic or
averaged sequences \cite{Bach14,Nemi09}, or only bound the expected
function value gap \cite{Atti26}, often for finite families. In the
following theorem, we establish almost sure rates of convergence
for the objective function gap with an arbitrary family of
functions, general convexity assumptions, and a general variance
control strategy.

\begin{theorem}
\label{t:sgd}
Let $(\mathsf{K},\EuScript{K})$ be a
measurable space and let $\upbeta\in\RPP$. For every
$\kS\in\mathsf{K}$, let $\fS_{\kS}\colon\HS\to\RR$ be a convex 
differentiable function such that $\nabla\fS_{\kS}$ is
$\upbeta$-Lipschitzian. Further, let
$k\colon(\upOmega,\FE,\PP)\to(\mathsf{K},\EuScript{K})$ be a random
variable and set 
\begin{equation}
\label{e:i}
(\forall\xS\in\HS)\quad\fS(\xS)
=\Int_{\upOmega}\fS_{k(\upomega)}(\xS)\PP(d\upomega).
\end{equation}
Assume that $\Argmin\fS\neq\emp$ and that there exists an
increasing function $\uppsi\colon\RP\to\RP$ such that
\begin{equation}
\label{e:99}
(\forall\xS\in\HS)\quad
\int_{\upOmega}\norm1{\nabla\fS_{k(\upomega)}(\xS)}^2
\PP(d\upomega)
\leq\uppsi\brk1{\norm{\xS}}.
\end{equation}
Let $x_{0}\in L^2(\upOmega,\FE,\PP;\HS)$, let
$\mathsf{p}\in\left]2/3,1\right]$, and set $(\forall\nnn)$
$\upgamma_{\nS}=(\nS+1)^{-\mathsf{p}}$. Iterate
\begin{equation}
\label{e:sgd1}
\begin{array}{l}
\textup{for}\;\nS=0,1,\ldots\\
\left\lfloor
\begin{array}{l}
k_{\nS}\;\textup{is a copy of}\;k\;\textup{and it is independent
of}\;\FE_{\nS}\\
x_{\nS+1}=x_{\nS}-\upgamma_{\nS}\nabla\fS_{k_{\nS}}(x_{\nS}).
\end{array}
\right.\\
\end{array}
\end{equation}
Then, $\Pas$,
\begin{equation}
\fS(x_{\nS})-\inf\fS(\HS)
=
\begin{cases}
\mathrm{o}\brk2{\dfrac{1}{\mathsf{n}^{1-\mathsf{p}}}},&\text{if}\;
\mathsf{p}\in\left]2/3,1\right[;\\[3mm]
\mathrm{o}\brk2{\dfrac{1}{\ln(\mathsf{n})}},
&\text{if}\;\mathsf{p}=1.
\end{cases}
\end{equation}
\end{theorem}
\begin{proof}
First, we note that $\sum_{\nnn}\upgamma_{\nS}=\pinf$ and
$\sum_{\nnn}\upgamma_{\nS}^2<\pinf$.
Moreover, as a consequence of \eqref{e:99} and the dominated
convergence theorem \cite[Proposition~1.2.5]{Hyto16},
\begin{equation}
(\forall\xS\in\HS)\quad\nabla\fS(\xS)
=\Int_{\upOmega}\nabla\fS_{k(\upomega)}(\xS)\PP(d\upomega).
\end{equation}
Let us now show that $(x_{\nS})_{\nnn}$ is bounded $\Pas$
Let $\zS\in\Argmin\fS$ and let $\nnn$. Then it
follows from \eqref{e:sgd1} and \eqref{e:99} that 
\begin{align}
\label{e:104}
&{\norm{x_{\nS+1}-\zS}^2}\nonumber\\
&\qquad=\norm1{\brk{x_{\nS}-\zS}
-\upgamma_{\nS}\nabla\fS_{k_{\nS}}(x_{\nS})}^2\nonumber\\
&\qquad=\norm{x_{\nS}-\zS}^2
-2\upgamma_{\nS}\scal1{x_{\nS}-\zS}
{\nabla\fS_{k_{\nS}}(x_{\nS})}
+\upgamma_{\nS}^2{\norm1{\nabla\fS_{k_{\nS}}(x_{\nS})}^2}
\nonumber\\
&\qquad\leq\norm{x_{\nS}-\zS}^2
-2\upgamma_{\nS}\scal1{x_{\nS}-\zS}
{\nabla\fS_{k_{\nS}}(x_{\nS})}
+2\upgamma_{\nS}^2{\norm1{\nabla\fS_{k_{\nS}}(x_{\nS})
-\nabla\fS_{k_{\nS}}(\zS)}^2}
+2\upgamma_{\nS}^2{\norm1{\nabla\fS_{k_{\nS}}(\zS)}^2}
\nonumber\\
&\qquad\leq\norm{x_{\nS}-\zS}^2
-2\upgamma_{\nS}\scal1{x_{\nS}-\zS}
{\nabla\fS_{k_{\nS}}(x_{\nS})}
+2\upbeta^2\upgamma_{\nS}^2{\norm{x_{\nS}-\zS}^2}
+2\upgamma_{\nS}^2{\norm1{\nabla\fS_{k_{\nS}}(\zS)}^2}
\nonumber\\
&\qquad=\brk1{1+2\upbeta^2\upgamma_{\nS}^2}
\norm{x_{\nS}-\zS}^2
-2\upgamma_{\nS}\scal1{x_{\nS}-\zS}
{\nabla\fS_{k_{\nS}}(x_{\nS})}
+2\upgamma_{\nS}^2{\norm1{\nabla\fS_{k_{\nS}}(\zS)}^2}
\;\;\Pas
\end{align}
We note that $x_{\nS}-\zS$ is $\FE_{\nS}$-measurable,
$\nabla\fS(\zS)=0$, and
$\upbeta^{-1}\nabla\fS$ is firmly nonexpansive
\cite[Corollary~18.17]{Livre1}. Therefore, \eqref{e:104},
\cite[Proposition~2.6.31]{Hyto16}, the
identity $\nabla\fS(\zS)=0$, and \eqref{e:99} yield
\begin{align}
&\EC1{\norm{x_{\nS+1}-\zS}^2}{\FE_{\nS}}
\nonumber\\
&\qquad\leq\brk1{1+2\upbeta^2\upgamma_{\nS}^2}
\norm{x_{\nS}-\zS}^2
-2\upgamma_{\nS}\EC2{\scal1{x_{\nS}-\zS}
{\nabla\fS_{k_{\nS}}(x_{\nS})}}{\FE_{\nS}}
+2\upgamma_{\nS}^2
\EC1{\norm{\nabla\fS_{k_{\nS}}(\zS)}^2}
{\FE_{\nS}}\nonumber\\
&\qquad\leq\brk1{1+2\upbeta^2\upgamma_{\nS}^2}
\norm{x_{\nS}-\zS}^2
-2\upgamma_{\nS}\scal2{x_{\nS}-\zS}
{\EC1{\nabla\fS_{k_{\nS}}(x_{\nS})}{\FE_{\nS}}}
+2\upgamma_{\nS}^2
\EC1{\norm{\nabla\fS_{k_{\nS}}(\zS)}^2}
{\FE_{\nS}}\nonumber\\
&\qquad=\brk1{1+2\upbeta^2\upgamma_{\nS}^2}
{\norm{x_{\nS}-\zS}^2}
-2\upgamma_{\nS}
\scal1{x_{\nS}-\zS}{\nabla\fS(x_{\nS})}
+2\upgamma_{\nS}^2\EE{\norm{\nabla\fS_{k}(\zS)}^2}
\nonumber\\
&\qquad\leq\brk1{1+2\upbeta^2\upgamma_{\nS}^2}
{\norm{x_{\nS}-\zS}^2}
-2\upbeta^{-1}\upgamma_{\nS}
\norm{\nabla\fS(x_{\nS})}^2
+2\upgamma_{\nS}^2\uppsi\brk1{\norm{\zS}}
\nonumber\\
&\qquad\leq\brk1{1+2\upbeta^2\upgamma_{\nS}^2}
{\norm{x_{\nS}-\zS}^2}
+2\upgamma_{\nS}^2\uppsi\brk1{\norm{\zS}}\;\;\Pas
\label{e:105}
\end{align}
Thus, it follows from Lemma~\ref{l:5} that 
$(\norm{x_{\nS}-\zS})_{\nnn}$ converges $\Pas$, which ensures that
$(x_{\nS})_{\nnn}$ is bounded $\Pas$ We infer from \eqref{e:sgd1}
and \cite[Proposition~17.7 and Theorem~18.15]{Livre1} that
\begin{equation}
\label{e:sgd2}
\upgamma_{\nS}\fS_{k_{\nS}}(x_{\nS+1})
\leq\upgamma_{\nS}\fS_{k_{\nS}}(\zS)
+\scal1{x_{\nS+1}-\zS}
{\upgamma_{\nS}\nabla\fS_{k_{\nS}}(x_{\nS})}
+\dfrac{\upbeta\upgamma_{\nS}}{2}\norm{x_{\nS+1}-x_{\nS}}^2\;\;\Pas
\end{equation}
Hence, \cite[Proposition~17.7]{Livre1} and \eqref{e:sgd1} yield
\begin{align}
&\scal1{x_{\nS+1}-\zS}{x_{\nS+1}-x_{\nS}}\nonumber\\
&\quad\leq\upgamma_{\nS}\brk1{\fS_{k_{\nS}}(\zS)
-\fS_{k_{\nS}}(x_{\nS})+\fS_{k_{\nS}}(x_{\nS})
-\fS_{k_{\nS}}(x_{\nS+1})} 
+\dfrac{\upbeta\upgamma_{\nS}}{2}\norm{x_{\nS+1}-x_{\nS}}^2
\nonumber\\
&\quad\leq\upgamma_{\nS}\brk2{\fS_{k_{\nS}}(\zS)
-\fS_{k_{\nS}}(x_{\nS})
+\scal1{x_{\nS}-x_{\nS+1}}{\nabla\fS_{k_{\nS}}(x_{\nS})}} 
+\dfrac{\upbeta\upgamma_{\nS}}{2}\norm{x_{\nS+1}-x_{\nS}}^2
\nonumber\\
&\quad=\upgamma_{\nS}\brk1{\fS_{k_{\nS}}(\zS)
-\fS_{k_{\nS}}(x_{\nS})}
+\upgamma_{\nS}^2\norm1{\nabla\fS_{k_{\nS}}(x_{\nS})}^2 
+\dfrac{\upbeta\upgamma_{\nS}}{2}\norm{x_{\nS+1}-x_{\nS}}^2\;\;\Pas
\label{e:109}
\end{align}
Consequently, we deduce from \eqref{e:99} that
\begin{align}
&\EC1{\norm{x_{\nS+1}-\zS}^2}{\FE_{\nS}}
\nonumber\\
&\;\;=\norm{x_{\nS}-\zS}^2
+2\EC1{\scal{x_{\nS}-\zS}{x_{\nS+1}-x_{\nS}}}{\FE_{\nS}}
+\EC1{\norm{x_{\nS+1}-x_{\nS}}^2}{\FE_{\nS}}
\nonumber\\
&\;\;=\norm{x_{\nS}-\zS}^2
+2\EC1{\scal{x_{\nS+1}-\zS}{x_{\nS+1}-x_{\nS}}}{\FE_{\nS}}
-\EC1{\norm{x_{\nS+1}-x_{\nS}}^2}{\FE_{\nS}}
\nonumber\\
&\;\;\leq\norm{x_{\nS}-\zS}^2
+2\upgamma_{\nS}\EC1{\fS_{k_{\nS}}(\zS)
-\fS_{k_{\nS}}(x_{\nS})}{\FE_{\nS}}+2\upgamma_{\nS}^2\EC1{
\norm{\nabla\fS_{k_{\nS}}(x_{\nS})}^2}{\FE_{\nS}}
\nonumber\\
&\quad\;\;+{\upbeta\upgamma_{\nS}}
\EC1{\norm{x_{\nS+1}-x_{\nS}}^2}{\FE_{\nS}}
-\EC1{\norm{x_{\nS+1}-x_{\nS}}^2}{\FE_{\nS}}
\nonumber\\
&\;\;=\norm{x_{\nS}-\zS}^2
+2\upgamma_{\nS}\brk1{\fS(\zS)
-\fS(x_{\nS})}+2\upgamma_{\nS}^2\EC1{
\norm{\nabla\fS_{k_{\nS}}(x_{\nS})}^2}{\FE_{\nS}}
\nonumber\\
&\quad\;\;+\brk1{{\upbeta\upgamma_{\nS}}-1}
\EC1{\norm{x_{\nS+1}-x_{\nS}}^2}{\FE_{\nS}}
\nonumber\\
&\;\;\leq\norm{x_{\nS}-\zS}^2
+2\upgamma_{\nS}\brk1{\fS(\zS)-\fS(x_{\nS})}
+2\upgamma_{\nS}^2\uppsi\brk1{\norm{x_{\nS}}}
+\brk1{\upbeta\upgamma_{\nS}-1}
\EC1{\norm{x_{\nS+1}-x_{\nS}}^2}{\FE_{\nS}}\nonumber\\
&\;\;\leq\norm{x_{\nS}-\zS}^2
+2\upgamma_{\nS}\brk1{\fS(\zS)-\fS(x_{\nS})}
+2\upgamma_{\nS}^2\uppsi\brk1{\norm{x_{\nS}}}\nonumber\\
&\quad\;\;+\max\brk[c]1{0,\upbeta\upgamma_{\nS}-1}
\EC1{\norm{x_{\nS+1}-x_{\nS}}^2}{\FE_{\nS}}\;\;\Pas
\label{e:110}
\end{align}
It results from the almost sure boundedness of $(x_{\nS})_{\nnn}$
and \eqref{e:99} that $(\uppsi(\norm{x_{\nS}}))_{\nnn}$ is bounded
$\Pas$ In addition, because $\sum_{\nnn}\upgamma_{\nS}^2<\pinf$, we
have $\upbeta\upgamma_{\nS}-1<0$ for $\nS$ large enough. We apply
Lemma~\ref{l:5} to \eqref{e:110} to obtain that 
\begin{equation}
\label{e:111}
\sum_{\nnn}\upgamma_{\nS}
\brk1{\fS(x_{\nS})-\inf\fS(\HS)}
<\pinf\;\;\Pas
\end{equation}
On the other hand, we deduce from
\cite[Theorem~18.15]{Livre1} that
\begin{align}
(\forall\nnn)\quad\fS(x_{\nS+1})
&\leq\fS(x_{\nS})
+\scal1{x_{\nS+1}-x_{\nS}}{\nabla\fS(x_{\nS})}
+\dfrac{\upbeta}{2}\norm1{x_{\nS+1}-x_{\nS}}^2\nonumber\\
&=\fS(x_{\nS})
-\upgamma_{\nS}\scal1{\nabla\fS_{k_{\nS}}(x_{\nS})}
{\nabla\fS(x_{\nS})}
+\dfrac{\upbeta\upgamma_{\nS}^2}{2}
\norm1{\nabla\fS_{k_{\nS}}(x_{\nS})}^2\;\;\Pas
\end{align}
Upon subtracting the infimum value and taking the conditional
expectation, we obtain 
\begin{align}
(\forall\nnn)\quad
&\EC1{\fS(x_{\nS+1})-\inf\fS(\HS)}
{\FE_{\nS}}\nonumber\\
&\leq\fS(x_{\nS})-\inf\fS(\HS)
-\upgamma_{\nS}\norm1{\nabla\fS(x_{\nS})}^2
+\dfrac{\upbeta\upgamma_{\nS}^2}{2}
\EC2{\norm1{\nabla\fS_{k_{\nS}}(x_{\nS})}^2}{\FE_{\nS}}
\nonumber\\
&\leq\fS(x_{\nS})-\inf\fS(\HS)
+\dfrac{\upbeta\upgamma_{\nS}^2}{2}\uppsi\brk1{\norm{x_{\nS}}}
\nonumber\\
&\leq\fS(x_{\nS})-\inf\fS(\HS)
+\dfrac{\upbeta\upgamma_{\nS}^2}{2}
\sup_{\jjj}\uppsi\brk1{\norm{x_{\jS}}}\;\;\Pas
\end{align}
To conclude, let us apply Theorem~\ref{t:1}\ref{t:1i} to
\begin{equation}
(\forall\nnn)\quad
\begin{cases}
\alpha_{\nS}=\upgamma_{\nS}\\
\xi_{\nS}=\fS(x_{\nS})-\inf\fS(\HS)\\
\delta_{\nS}=0\\
\varepsilon_{\nS}=\dfrac{\upbeta\upgamma_{\nS}^2}{2}
\displaystyle\sup_{\jjj}\uppsi\brk1{\norm{x_{\jS}}}.
\end{cases}
\end{equation}
Since $\sum_{\nnn}\upgamma_{\nS}=\pinf$ and \eqref{e:111}
holds, it remains to check that
$\sum_{\nnn}\uptheta_{\nS}\varepsilon_{\nS}<\pinf\;\Pas$ We have
\begin{align}
\sum_{\nnn}\uptheta_{\nS}\varepsilon_{\nS}
&=\dfrac{\upbeta}{2}
\displaystyle\sup_{\jjj}\uppsi\brk1{\norm{x_{\jS}}}
\sum_{\nnn}\brk3{\upgamma_{\nS}^2\sum_{\kS=0}^{\nS}\upgamma_{\kS}}
\nonumber\\
&=\dfrac{\upbeta}{2}
\displaystyle\sup_{\jjj}\uppsi\brk1{\norm{x_{\jS}}}
\sum_{\nnn}\brk3{(\nS+1)^{-2\mathsf{p}}
\sum_{\kS=0}^{\nS}(\kS+1)^{-\mathsf{p}}}
\label{e:b1}
\end{align}
We consider the two cases of interest.
\begin{itemize}
\item
Suppose that $\mathsf{p}\in\left]2/3,1\right[$. Then
we derive from \eqref{e:b1} that
\begin{align}
\sum_{\nnn}\uptheta_{\nS}\varepsilon_{\nS}
&\leq\dfrac{\upbeta}{2}
\displaystyle\sup_{\jjj}\uppsi\brk1{\norm{x_{\jS}}}
\sum_{\nnn}\brk3{(\nS+1)^{-2\mathsf{p}}
\brk2{1+\dfrac{(\nS+1)^{1-\mathsf{p}}-1}{1-\mathsf{p}}}}
\nonumber\\
&=\dfrac{\upbeta}{2}
\displaystyle\sup_{\jjj}\uppsi\brk1{\norm{x_{\jS}}}
\sum_{\nnn}\brk3{
\dfrac{-\mathsf{p}}{1-\mathsf{p}}(\nS+1)^{-2\mathsf{p}}
+\dfrac{1}{1-\mathsf{p}}(\nS+1)^{1-3\mathsf{p}}
}
\nonumber\\
&<\pinf\;\;\Pas
\end{align}
\item
Suppose that $\mathsf{p}=1$. Then
\eqref{e:b1} yields
\begin{equation}
\sum_{\nnn}\uptheta_{\nS}\varepsilon_{\nS}
\leq\dfrac{\upbeta}{2}
\displaystyle\sup_{\jjj}\uppsi\brk1{\norm{x_{\jS}}}
\sum_{\nnn}\brk3{(\nS+1)^{-2}
\brk1{1+\ln(\nS+1)}}<\pinf\;\;\Pas
\end{equation}
\end{itemize}
In both cases, 
$\sum_{\nnn}\uptheta_{\nS}\varepsilon_{\nS}<\pinf\;\Pas$ and
we infer from Theorem~\ref{t:1}\ref{t:1i} that, $\Pas$,
\begin{equation}
\xi_{\nS+1}=\fS(x_{\nS+1})-\inf\fS(\HS)=
\mathrm{o}\brk2{\dfrac{1}{\uptheta_{\nS}}}=
\begin{cases}
\mathrm{o}\brk2{\dfrac{1}{(\nS+1)^{1-\mathsf{p}}}},&\text{if}\;
\mathsf{p}\in\left]2/3,1\right[;\\[3mm]
\mathrm{o}\brk2{\dfrac{1}{\ln(\nS+1)}},
&\text{if}\;\mathsf{p}=1,
\end{cases}
\end{equation}
as desired.
\end{proof}


\begin{thebibliography}{99}
\setlength{\itemsep}{0pt}
\small

\bibitem{Abel28}
N. H. Abel,
Note sur le m\'emoire de Mr. L. Olivier No. 4 du second tome de ce
journal, ayant pour titre ,,remarques sur les s\'eries infinies et
leur convergence'', suivie d'une note de Mr. L. Olivier sur le
m\^eme objet,
{\em J. Reine Angew. Math.},
vol. 3, pp. 79--82, 1828.

\bibitem{Atti26}
A. Attiaq, M. Schliserman, U. Sherman, and T. Koren,
Fast last-iterate convergence of SGD in the smooth interpolation
regime,
{\em Adv. Neural Inf. Process. Syst.},
vol. 38, pp. 104951--104987, 2026.

\bibitem{Bach14}
F. Bach,
Adaptivity of averaged stochastic gradient descent to local strong
convexity for logistic regression,
{\em J. Mach. Learn. Res.},
vol. 15, pp. 595--627, 2014.

\bibitem{Livre1} 
H. H. Bauschke and P. L. Combettes, 
{\em Convex Analysis and Monotone Operator Theory in Hilbert 
Spaces}, 2nd ed. 
Springer, New York, 2017.

\bibitem{Boas65}
R. P. Boas, Jr.,
Quasi-positive sequences and trigonometric series,
{\em Proc. London Math. Soc.}, 
vol. s3-14A, pp. 38--46, 1965. 

\bibitem{Brav24}
M. Bravo and R. Cominetti, 
Stochastic fixed-point iterations for nonexpansive maps: 
Convergence and error bounds,
{\em SIAM J. Control Optim.},
vol. 62, pp. 191--219, 2024.

\bibitem{Brez78} 
H. Br\'ezis and P. L. Lions,
Produits infinis de r\'esolvantes,
{\em Israel J. Math.},
vol. 29, pp. 329--345, 1978.

\bibitem{Svva21}
M. N. B\`ui and P. L. Combettes,
Bregman forward-backward operator splitting, 
{\em Set-Valued Var. Anal.},
vol. 29, pp. 583--603, 2021. 

\bibitem{Acnu24} 
P. L. Combettes,
The geometry of monotone operator splitting methods,
{\em Acta Numer.},
vol. 33, pp. 487--632, 2024.

\bibitem{Moco26}
P. L. Combettes and J. I. Madariaga,
A geometric framework for stochastic iterations,
{\em Math. Comput.}, 
\url{https://doi.org/10.1090/mcom/4230}

\bibitem{Sadd25} 
P. L. Combettes and J. I. Madariaga,
Asymptotic analysis of an abstract stochastic scheme for solving
monotone inclusions,
arxiv, 2025. 
\url{https://arxiv.org/pdf/2512.03023}

\bibitem{Siop15} 
P. L. Combettes and J.-C. Pesquet, 
Stochastic quasi-Fej\'er block-coordinate fixed point iterations
with random sweeping,
{\em SIAM J. Optim.}, 
vol. 25, pp. 1221--1248, 2015.

\bibitem{MaPa18}
P. L. Combettes, S. Salzo, and S. Villa, 
Consistent learning by composite proximal thresholding,
{\em Math. Program.},
vol. B167, pp. 99--127, 2018. 

\bibitem{Davi16}
D. Davis and W. Yin, 
Convergence rate analysis of several splitting schemes,
in: R. Glowinski, S. J. Osher, and W. Yin (eds.),
{\em Splitting Methods in Communication, Imaging, Science, and
Engineering}, pp. 115--163. Springer, New York, 2016.

\bibitem{Dini67}
U. Dini,
Sulle serie a termini positivi,
{\em Giornale di Matematiche (G. Battaglini)},
vol. VI, pp. 166--175, 1868.

\bibitem{Dong22} 
Q.-L. Dong, Y. J. Cho, S. He, P. M. Pardalos, and T. M. Rassias, 
{\em The Krasnosel'ski\u\i--Mann Iterative Method -- Recent
Progress and Applications}.
Springer, New York, 2022.

\bibitem{Dufl90} 
M. Duflo,
{\em M\'ethodes R\'ecursives Al\'eatoires}.
Masson, Paris, 1990. English translation: 
{\em Random Iterative Models}.
Springer, New York, 1997.

\bibitem{Ermo69} 
Yu. M. Ermol'ev, 
On the method of generalized stochastic gradients and 
quasi-Fej\'er sequences,
{\em Cybernetics}, 
vol. 5, pp. 208--220, 1969.

\bibitem{Fran22} 
B. Franci and S. Grammatico,
Convergence of sequences: A survey,
{\em Annu. Rev. Control},
vol. 53, pp. 161--186, 2022.

\bibitem{Glad65}
E. G. Gladyshev,
On stochastic approximation,
{\em Theory Probab. Appl.},
vol. 10, pp 275--278, 1965.

\bibitem{Groe72} 
C. W. Groetsch, 
A note on segmenting Mann iterates,
{\em J. Math. Anal. Appl.},
vol. 40, pp. 369--372, 1972.

\bibitem{Hyto16}
T. Hyt\"onen, J. van Neerven, M. Veraar, and L. Weis,
{\em Analysis in Banach Spaces. Volume I: Martingales and 
Littlewood--Paley Theory}.
Springer, New York, 2016.

\bibitem{Lian16}
J. Liang, J. Fadili, and G. Peyr\'e,
Convergence rates with inexact non-expansive operators,
{\em Math. Program.},
vol. A159, pp. 403--434, 2016.

\bibitem{Liu24}
J. Liu and Y. Yuan,
Almost sure convergence rates analysis and saddle avoidance of
stochastic gradient methods,
{\em J. Mach. Learn. Res.},
vol. 25, pp. 1--40, 2024.

\bibitem{Nemi09}
A. Nemirovski, A. Juditsky, G. Lan, and A. Shapiro,
Robust stochastic approximation approach to stochastic programming,
{\em SIAM J. Optim.},
vol. 19, pp. 1574--1609, 2009.

\bibitem{Neri26}
M. Neri and T. Powell,
A quantitative Robbins--Siegmund theorem, 
{\em Ann. Appl. Probab.},
vol. 36, pp. 636--651, 2026.

\bibitem{Oliv27}
L. Olivier, 
Remarques sur les s\'eries infinies et leur convergence,
{\em J. Reine Angew. Math.},
vol. 2, pp. 31--44, 1827.

\bibitem{Robb71} 
H. Robbins and D. Siegmund, 
A convergence theorem for non negative almost supermartingales 
and some applications, in:
{\em Optimizing Methods in Statistics}, (J. S. Rustagi, Ed.), 
pp. 233--257. Academic Press, New York, 1971.

\bibitem{Robe68}
M. M. Robertson,
A generalization of quasi-monotone sequences,
{\em Proc. Edinb. Math. Soc.},
vol. 16, pp. 37--41, 1968.

\bibitem{Sing78}
N. Singh and K. M. Sharma,
Convergence of certain cosine sums in a metric space -- L,
{\em Proc. Amer. Math. Soc.},
vol. 72, pp. 117--120, 1978.

\bibitem{Szas48}
O. Sz\'asz,
Quasi-monotone series,
{\em Amer. J. Math.},
vol. 70, pp. 203--206, 1948.

\bibitem{Vall14}
Ch.-J. de la Vall\'ee Poussin, 
{\em Cours d'Analyse Infinit\'esimale}, tome I, 3e \'ed. 
Gauthier--Villars, Paris, 1914.

\end{thebibliography}
\end{document}